\documentclass[onefignum,onetabnum]{siamart220329}
\usepackage[centering]{geometry}
\usepackage{graphicx}
\usepackage{epstopdf}
\usepackage{multirow}
\usepackage{todonotes}
\usepackage{caption}
\usepackage{amsmath}
\usepackage{amsfonts}
\usepackage{mathtools}
\usepackage{enumerate}
\usepackage{enumitem}
\usepackage{thmtools,thm-restate}

\usepackage[noend]{algpseudocode}

\newsiamthm{assumption}{Assumption}
\newsiamremark{remark}{Remark}
\newsiamremark{example}{Example}

\ifpdf
\DeclareGraphicsExtensions{.eps,.pdf,.png,.jpg}
\else
\DeclareGraphicsExtensions{.eps}
\fi

\renewcommand{\paragraph}[1]{\textbf{#1.}}

\DeclareMathOperator*{\argmin}{arg\,min}

\newcommand{\bm}{\boldsymbol}

\newcommand{\dd}{{\mathrm{d}}}

\newcommand{\R}{\mathbb{R}}
\newcommand{\N}{\mathbb{N}}

\newcommand{\bj}{{\bm j}}
\newcommand{\bk}{{\bm k}}

\newcommand{\bx}{{\bm x}}
\newcommand{\bX}{{\bm X}}
\newcommand{\by}{\bm y}
\newcommand{\bY}{\bm Y}

\newcommand{\bu}{{\bm u}}
\newcommand{\bU}{{\bm U}}
\newcommand{\bxi}{\bm\xi}
\newcommand{\bXi}{\bm\Xi}

\newcommand{\mM}{\mathsf M}

\newcommand{\mC}{\mathsf C}

\newcommand{\mL}{\mathsf L}

\renewcommand{\ldots}{\makebox[1em][c]{.\hfil.\hfil.}}

\newcommand{\lpx}[1]{L_\lambda^{#1}}

\newcommand{\lpnormx}[2]{ \big\| {#1} \big\|_{\lpx{#2}(\mathcal{X})} }

\newcommand{\indi}{\mathtt{1}}
\renewcommand{\d}{\mathrm{d}}

\newcommand{\lowsup}[1]{{\raisebox{-2pt}{\scriptsize$#1$}}}

\usepackage{tikz}
\usetikzlibrary{positioning,shapes,calc,3d}
\usepackage{pgfplots}
\usepgfplotslibrary{fillbetween}

\pgfplotsset{%
every axis/.append style={width=.45\textwidth,
                              axis x line=bottom, axis y line=left,
                              x axis line style={thick,->}, y axis line style={thick,->},
                              tick align=inside, tick style={thick},
                              every x tick label/.style={font=\footnotesize},
                              every y tick label/.style={font=\footnotesize},
                              },
every axis legend/.append style={
                              legend columns=1,
                              font=\footnotesize,
                              draw=none,
                              fill=white,
                              },
every axis x label/.style={at={(0.5,0.1)},below,fill=none,fill opacity=1,text opacity=1},
every axis y label/.style={at={(0.1,0.5)},fill=none,fill opacity=1,text opacity=1,rotate=90},
compat=newest,
}

\title{Self-reinforced polynomial approximation methods for concentrated probability densities}

\author{Tiangang Cui\thanks{School of Mathematics, Monash University, Victoria 3800, Australia \email{tiangang.cui@monash.edu}}
\and Sergey Dolgov\thanks{Department of Mathematical Sciences, University of Bath, Bath, UK \email{s.dolgov@bath.ac.uk}}
\and Olivier Zahm\thanks{Univ. Grenoble Alpes, Inria, CNRS, Grenoble INP, LJK, Grenoble, France \email{olivier.zahm@inria.fr}}
}

\begin{document}

\maketitle

\begin{abstract}
Transport map methods offer a powerful statistical learning tool that can couple a target high-dimensional random variable with some reference random variable using invertible transformations. This paper presents new computational techniques for building the Knothe--Rosenblatt (KR) rearrangement based on general separable functions. We first introduce a new construction of the KR rearrangement---with guaranteed invertibility in its numerical implementation---based on approximating the density of the target random variable using tensor-product spectral polynomials and downward closed sparse index sets. Compared to other constructions of KR arrangements based on either multi-linear approximations or nonlinear optimizations, our new construction only relies on a weighted least square approximation procedure. Then, inspired by the recently developed deep tensor trains  (Cui and Dolgov, \emph{Found. Comput. Math.} 22:1863--1922, 2022), we enhance the approximation power of sparse polynomials by preconditioning the density approximation problem using compositions of maps. This is particularly suitable for high-dimensional and concentrated probability densities commonly seen in many applications. We approximate the complicated target density by a composition of \emph{self-reinforced} KR rearrangements, in which previously constructed KR rearrangements---based on the same approximation ansatz---are used to precondition the density approximation problem for building each new KR rearrangement. We demonstrate the efficiency of our proposed methods and the importance of using the composite map on several inverse problems governed by ordinary differential equations (ODEs) and partial differential equations (PDEs).
\end{abstract}



\setlength{\abovedisplayskip}{4pt}
\setlength{\belowdisplayskip}{4pt}
\setlength{\intextsep}{4pt}
\setlength{\floatsep}{4pt}
\setlength{\textfloatsep}{4pt}
\allowdisplaybreaks


\section{Introduction}

Since the invention of the first general-purpose digital computer ENIAC, a common strategy for simulating complicated high-dimensional random variables is to either deterministically or stochastically transform reference random variables, e.g., Gaussian, into target random variables \cite{metropolis1987beginning}. Classical examples include Markov chain Monte Carlo (MCMC) and importance sampling. The former generates a Markov chain of random variables that asymptotically converges to the target random variable, while the latter characterizes target random variables using reference random variables and weights. The development of optimal transport \cite{villani2008optimal} and related transformation methods offers new insight into this task. Instead of building Markov chains or using weights, one can couple a target random variable $\bX$ taking value in $\mathcal{X} \subseteq \R^d$ with a reference random variable $\bU$ taking value in $\mathcal{U} \subseteq \R^d$  via a mapping  $\mathcal{T}:  \mathcal{U} \rightarrow \mathcal{X}$ such that $\mathcal{T}(\bU) = \bX$. 
Denoting the laws of $\bX$ and $\bU$ by $\nu_\bX$ and $\nu_\bU$, respectively, the pushforward of $\nu_\bU$ under the mapping $\mathcal{T}$, denoted by $\mathcal{T}_\sharp \, \nu_\bU$, is equivalent to $\nu_\bX$.
In this work, we focus on the KR rearrangement \cite{knothe1957contributions,rosenblatt1952remarks}, which is a particular way of constructing such a mapping in the triangular form of
\[
    \mathcal{T}(\bU) =  \left(\begin{array}{l} \mathcal{T}_1(U_1) \\ \mathcal{T}_2(U_1, U_2) \vspace{-4pt}\\ \vdots \\  \mathcal{T}_d(U_1, U_2, \ldots, U_d)  \end{array}\right) .
\]
Here each term $\smash{\mathcal{T}_k : \R^k \rightarrow \R}$ is a scalar-valued multivariate function depending on the first $k$ input variables, so that the image of $(U_1, \ldots, U_k)$ by the first $k$ terms of the KR rearrangement has the same law of the marginal target variable $(X_1, \ldots, X_k)$.

One class of methods for computing the KR rearrangement adopts a \emph{map-from-samples} approach that estimates the map using a set of samples drawn from the target measure $\nu_\bX$. This is typically used in contexts where we do not have access to the target measure, but where realizations of $\bX$ are available. Notable examples include methods presented in \cite{baptista2020adaptive,parno2018transport}, which identify the map $\mathcal{T}$ by minimizing the Kullback--Leibler (KL) divergence of $\nu_\bX$ from $\mathcal{T}_\sharp \, \nu_\bU$ over the family of triangular transformations parametrized by polynomial bases. In this setting, the KL divergence $D_{\rm KL}(\nu_\bX \| \mathcal{T}_\sharp \, \nu_\bU)$ is approximated by the Monte Carlo method using samples drawn from $\nu_\bX$. The normalizing flow methods, e.g., \cite{caterini2021variational,chen2019residualflows,kruse2019hint,papamakarios2021normalizing}, adopt a similar KL minimization problem, but the maps may not be triangular and are often parametrized by neural networks. By and large, the map-from-samples approach is flexible to implement, as it only requires a set of samples drawn from $\nu_\bX$. However, it comes with an $N^{-1/2}$ error rate, where $N$ is the sample size, due to the Monte Carlo estimate of the KL divergence. See \cite{wang2022minimax} and references therein for the analysis.

Another class of methods adopts a \emph{map-from-density} approach that estimates the map using point evaluations of the unnormalized density of $\bX$. This is suitable for situations where the target density can be evaluated, but sampling from it can be difficult. Notable examples include the work of \cite{bigoni2019greedy,el2012Bayesian,spantini2018inference} that finds a triangular map $\mathcal{T}$ by minimizing the \emph{reversed} KL divergence $D_{\rm KL}( \mathcal{T}_\sharp\,\nu_\bU \| \nu_\bX )$.
Although this approach may not suffer from the slow $N^{-1/2}$ error rate of the \emph{map-from-samples} approach, its training procedure is often quite involved in practice---the objective function presents many local minima and each optimization iteration requires repeated evaluations of the target density at transformed reference variables under the candidate map.

Instead of minimizing the reversed KL divergence, \emph{function approximation methods} have also been developed to construct KR rearrangements from unnormalized densities. The idea is to compute the \emph{exact} KR rearrangement of an \emph{approximate} density of the target random variable. For instance, tensor methods are used in \cite{cui2021deep,cui2021conditional,dafs-tt-bayes-2019} to construct the approximate density.
By doing so, this approach bypasses the highly nonlinear optimization procedure, and the associated error analysis is often well established.
In many applications, however, the high-dimensional random variables of interest have concentrated density functions and exhibit complicated nonlinear interactions. In this case, it is often computationally infeasible to directly approximate the target density function with standard approximation tools. Following  \cite{cui2021deep}, one way to circumvent this challenge is to introduce a layered construction technique that builds a composition of KR rearrangements, i.e., $\mathcal{T} = \mathcal{Q}^{(1)} \circ \cdots \circ \mathcal{Q}^{(L)}.$ Each map $\mathcal{Q}^{(\ell)}$ is greedily constructed as the KR rearrangement from the reference random variable $\bU$ to the preconditioned bridging random variable the law $(\mathcal{Q}^{(1)} \circ \cdots \circ \mathcal{Q}^{(\ell-1)})^\sharp\,\nu_{\bX^{(\ell)}}$. Here the bridging random variable $\bX^{(\ell)}$ is chosen such that its density function has lower complexity. The preconditioning is the key to reducing the complexity of the training of KR rearrangements.

\paragraph{Contributions and outline} In this paper, we present generalizations and improvements of the function approximation methods used to build the KR rearrangement from density. Instead of using tensor methods as in the seminal papers \cite{cui2021deep,cui2021conditional}, we propose in Section \ref{sec:sirt} to approximate the target density using a sparse tensor-product basis via a least square formulation. More specifically, we approximate the square root of the unnormalized target density to ensure the \emph{nonnegativity} of the density approximation by squaring the resulting function. The nonnegativity is essential for the numerical stability of the KR rearrangement.
Compared to the squared tensor methods, the computation complexity of KR rearrangements using sparse tensor-product bases is invariant to variable ordering.

While any tensor-product basis can be used in principle, we show in Section \ref{sec:adapt_sirt} that using orthogonal basis functions yields extra sparsity patterns in the computation of the resulting KR rearrangement. Moreover, we consider orthogonal polynomial bases in this work, as they provide analytical expressions that enable fast computations of the KR rearrangement. Borrowing ideas from \cite{migliorati2015adaptive,migliorati2019adaptive}, we also show how to use downward closed index sets and the optimal weighted least square method \cite{cohen2017optimal,hampton2015coherence,migliorati2019adaptive,narayan2017christoffel} to adaptively search for a suitable polynomial basis in high-dimension.

In Section \ref{sec:dirt}, we employ the layered construction technique of \cite{cui2021deep} to design a \emph{self-reinforced} sparse polynomial approximation procedure, where the composition of KR rearrangements built in previous layers is used to precondition the next density approximation problem, and thus leads to a new layer of KR rearrangement. This offers a more efficient KR rearrangement than the single-level construction in terms of both accuracy and complexity. It also offers insights into enriching the approximation power of sparse polynomial methods in general. We provide practical recipes for building composite KR rearrangements using adaptations and various forms of bridging densities. 
In Section \ref{sec:numerics}, we demonstrate the efficiency of our proposed methods on several Bayesian inverse problems governed by ODEs and PDEs. 


\section{Building KR rearrangements by separable function approximation}\label{sec:sirt}

\subsection{Notation}
We consider probability measures that are absolutely continuous with respect to the Lebesgue measure. The law of a random variable $\bX$ is denoted by $\nu_\bX$ and its density is denoted by $f_\bX$.
Given another random variable $\bY\sim \nu_{\bY}$ with density $f_{\bY}$, the Hellinger distance between $\nu_{\bX}$ and $\nu_{\bY}$ is defined by
\begin{equation}\label{eq:Hellinger}
 D_{\rm H}(\nu_{\bX}, \nu_{\bY} ) = \Big(\frac{1}{2} \int \left(\sqrt{f_{\bX}(\bx)}-\sqrt{f_{\bY}(\bx)}\right)^2\d \bx \Big)^\frac12.
\end{equation}

Given a diffeomorphism $\mathcal{S}: \mathcal{X} \rightarrow \mathcal{U}$, the law of $\mathcal{S}(\bX)$ is called the \emph{pushforward} of $\nu_\bX$ under $\mathcal{S}$ and is denoted by $\mathcal{S}_\sharp\,\nu_\bX$.
Similarly, the law of $\mathcal{S}^{-1}(\bU)$ is called the \emph{pullback} of $\nu_\bU$ under $\mathcal{S}$ and is denoted by $\mathcal{S}^\sharp \, \nu_\bU$. Both $\mathcal{S}_\sharp\,\nu_\bX$ and $\mathcal{S}^\sharp \, \nu_\bU$ admit densities given by the change of variable formula
\begin{align}
 \mathcal{S}_\sharp\, f_\bX(\bu)
 &= \big(f_\bX \circ \mathcal{S}^{-1}\big)(\bu) \, \big| \nabla_{\bu} \mathcal{S}^{-1}(\bu)\big| \\
 \mathcal{S}^\sharp\, f_\bU(\bx) &= \big(f_\bU \circ \mathcal{S} \big)(\bx) \, \big| \nabla_{\bx} \mathcal{S}(\bx)\big|, \label{eq:pullbackDensity}
\end{align}
for all $\bu,\bx\in\R^d$, where $|\cdot|$ denotes the determinant of a matrix.

We denote $\mathbb{N}$ the set of all positive integers and $\mathbb{N}_0 = \mathbb{N} \cup \{0\}$. For $d \in \mathbb{N}$, we use boldface letters to denote $d$-dimensional vectors and multi-indices. For example, $\bk = (k_1, k_2, \ldots, k_d) \in \mathbb{N}_0^d$ and $\bx = (x_1, x_2, \ldots, x_d) \in \R^d$. Also, let $\bm 0 = (0, 0, \ldots, 0) \in \mathbb{N}_0^d$, $\bm 1 = (1, 1, \ldots, 1) \in \mathbb{N}^d$, and let ${\bm e}_j  = (0, \ldots, 1, \ldots, 0)$ be the $j$-th unit vector in $\R^d$.
The cardinality of a set $\mathcal{K}$ is denoted by $|\mathcal{K}|$. 

Given a $d$-dimensional space, we denote the full coordinate index set by $[d] = \{1, \ldots, d\} \subset \mathbb{N}$. Given a subset of coordinate indices $\mathfrak{u} = (u_1, \ldots, u_s) \subseteq [d]$ with cardinality $|\mathfrak{u}| = s \leq d$, we let $\mathfrak{u}^c = [d] \setminus \mathfrak{u}$ be the complementary subset. This way, we can also define $\bx_{\mathfrak{u}} = ( x_{u_1}, \ldots, x_{u_s}) \in \R^{|\mathfrak{u}|}$, $\mathcal{X}_{\mathfrak{u}} = \mathcal{I}^{|\mathfrak{u}|}$. For any scalar-value function $h:\R\rightarrow\R$, we write $h(\bx_{\mathfrak{u}}) = \prod_{i \in \mathfrak{u}} h(x_i)$.

Given a matrix $\mathsf{A} \in \R^{n\times d}$ and coordinate index sets $\mathfrak{u} \subseteq [n]$, $\mathfrak{v} \subseteq [d]$ with cardinalities $|\mathfrak{u}| = s \leq n$, $|\mathfrak{v}| = t \leq d$.  We let $\mathsf{A}_{\mathfrak{u}:} \in \R^{s \times d}$ and $\mathsf{A}_{:\mathfrak{v}} \in \R^{n \times t}$ be the matrices consisting of a subset of rows of $\mathsf{A}$ indexed by $\mathfrak{u}$ and a subset of columns of $\mathsf{A}$ indexed by $\mathfrak{v}$, respectively. Similarly,  $\mathsf{A}_{\mathfrak{u}\mathfrak{v}} \in \R^{s \times t}$ represents the submatrix  of $\mathsf{A}$ with  rows  indexed by $\mathfrak{u}$ and columns indexed by $\mathfrak{v}$. 

For a multi-index set $\mathcal{K} \subset \mathbb{N}_0^d$, we define the function $\verb|matrix|(\mathcal{K})$ to convert $\mathcal{K}$ into a matrix $\mathsf{K} = \verb|matrix|(\mathcal{K})$ such that $\mathsf{K} \in \mathbb{N}_0^\lowsup{|\mathcal{K}| \times d}$. We also introduce a mapping $\sigma : \mathcal{K} \rightarrow [|\mathcal{K}|]$, where $[|\mathcal{K}|] = \{1, 2, \ldots, |\mathcal{K}|\}$, such that $\sigma(\bk) \in [|\mathcal{K}|]$ gives the scalar-valued index of the multi-index $\bk$ the in the ordered set $\mathcal{K}$. Under this notation, for any $\bk\in\mathcal{K}$, the row $\mathsf{K}_{\sigma(\bk):} = (k_1,\hdots,k_d)$ contains the integers of the multi-index $\bk$.

All multi-dimensional spaces used in this paper are Cartesian products spaces equipped with product-form, normalized weight functions. For example, for $\mathcal{X} \subseteq \R^d$ we have  $\mathcal{X} =  \mathcal{I}^d$ where $\mathcal{I} \subseteq \R$ and the associated weight function takes the form $\lambda(\bx) = \prod_{i = 1}^{d} \lambda(x_i)$ such that $\int_{\mathcal{I}} \lambda(x_i) \dd x_i = 1$ for $i = 1, \ldots, d$. On the interval $\mathcal{I}$, we denote the inner product and weighted inner product by
\[
\langle f,g\rangle = \int_\mathcal{I} f(x)g(x) \dd x \quad \text{and} \quad \langle f,g\rangle_{\lambda} = \int_\mathcal{I} f(x)g(x) \lambda(x)\dd x,    
\]
respectively. The norms induced by these inner products are denoted by $\|f\|_{L^2(\mathcal{X})}:=\langle f,f\rangle^{1/2}$ and $\|f\|_{L_\lambda^2(\mathcal{X})}:=\langle f,f\rangle_{\lambda}^{1/2}$.

\subsection{KR rearrangement}\label{sec:KR}
The basic idea behind the KR rearrangement is the fundamental property of the distribution function: given a \emph{scalar} random variable $X$ with density $f_X$, its distribution function $\mathcal{F}_X(x)=\int_{-\infty}^x f_X(x^\prime)\d x^\prime$ maps the random variable $X$ into a random variable uniformly distributed on the interval $[0,1]$. To extend this to \emph{vector-valued} random variables, we factorize the multivariate density function $f_\bX$ as the product of one-dimensional conditional marginals
\begin{equation}\label{eq:productOfCondMarg}
 f_\bX(\bx) = f_{X_1}(x_1) f_{X_2|X_1}(x_2|x_1) \hdots f_{X_d|\bX_{[d-1]}}(x_d|\bx_{[d-1]}),
\end{equation}
where the conditional marginals and the marginals are respectively defined by
\begin{align}
 f_{X_t|\bX_{[t-1]}}(x_t|\bx_{[t-1]}) &= \frac{ f_{\bX_{[t]}}(\bx_{[t-1]},x_t)}{ f_{\bX_{[t-1]}}(\bx_{[t-1]}) }, \\
 f_{\bX_{[t]}}(\bx_{[t]}) &= \int f_{\bX}(\bx_{[t]},\bx_{[t]^c}^\prime) \d \bx_{[t]^c}^\prime .\label{eq:marginal111}
\end{align}
Following \cite{rosenblatt1952remarks}, we let $\mathcal{F}_{\bX}: \mathcal{X} \mapsto [0,1]^d$ be defined as
\begin{equation}
    \mathcal{F}_{ \bX}(\bx) =
    \left(\begin{array}{l}
     \mathcal{F}_{ X_1}(x_1 ) \\
     \mathcal{F}_{ X_{2} |  X_{1}}(x_2 | x_{1}) \vspace{-4pt}\\
     \vdots \\
      \mathcal{F}_{ X_{d} |  \bX_{[d{-}1]}}(x_d | \bx_{[d{-}1]})
    \end{array}\right),
\label{eq:Rosenblatt111}
\end{equation}
where each components are the univariate conditional distribution function
\begin{equation}
 \mathcal{F}_{ X_{t} |  \bX_{[t{-}1]}}(x_t | \bx_{[t{-}1]}) = \int_{-\infty}^{x_t}  f_{ X_{t} | \bX_{[t{-}1]}}(x_t^\prime | \bx_{[t{-}1]}) \dd x_t^\prime .\label{eq:conditional111}
\end{equation}
By non-negativity of the conditional marginal, each component $\mathcal{F}_{ X_{t} |  \bX_{[t{-}1]}}$ is an non-decreasing function of the $t$-th variable, meaning $\partial_t \mathcal{F}_{ X_{t} |  \bX_{[t{-}1]}}(x_t | \bx_{[t{-}1]})\geq0$. In addition, $\mathcal{F}_{\bX}$ is \emph{lower-triangular} in the sense that the $t$-th component is a scalar-valued function depending on only the first $t$ variables $\bx_{[t]}$.
The map $\mathcal{F}_{\bX}$ satisfies 
\[
    \mathcal{F}_{\bX}^\sharp  f^{}_{\bXi}(\bx)
    = \big(f_{\bXi} \circ \mathcal{F}_{\bX} \big)(\bx)\,\big| \nabla_{\bx} \mathcal{F}_{\bX}(\bx) \big|  = \big| \nabla_{\bx} \mathcal{F}_{\bX}(\bx) \big|
    = f_{\bX}(\bx),
\]
where $\bXi$ is the uniform random variable on $[0,1]^d$.
This shows that $\bXi = \mathcal{F}_{\bX}(\bX)$, which is equivalent to $\bX = \mathcal{F}_{\bX}^{-1}(\bXi)$. The inverse map $\mathcal{F}_{\bX}^{-1}$ is also \emph{lower-triangular} and $\bx=\mathcal{F}_{\bX}^{-1}(\bxi)$ can be evaluated as
\begin{equation}\label{eq:inverse_trans}
\mathcal{F}_{\bX}^{-1}(\bxi) 
= \big[
\mathcal{F}_{X_1}^{-1}(\xi_1) , \;
\mathcal{F}_{X_2|X_1}^{-1}(\xi_2|x_1), \; \cdots, \;
\mathcal{F}_{ X_{d} | \bX_{[d{-}1]}}^{-1}(\xi_d | \bx_{[d{-}1]})\big]{}^\top,
\end{equation}
where $\mathcal{F}_{\smash{X_t|\bX_{[t-1]}}}^{-1}\!(\cdot|\bx_{[t-1]})$ denotes the inverse function of $x_t\mapsto\mathcal{F}_{\smash{X_t|\bX_{[t-1]}}}(x_t|\bx_{[t-1]})$. We assume that the conditional marginals $f_{X_t|\bX_{[t-1]}}$ are positive almost surely, so that $x_t\mapsto\mathcal{F}_{X_t|\bX_{[t-1]}}(x_t|\bx_{t-1})$ is strictly monotone and thus admits an inverse function.

\begin{remark}
The above construction leads to a KR rearrangement $\mathcal{T}_\bX(\bU) = \bX$ with a uniform reference variable $\bU$. To use a general reference random variable $\bU$, we can define the composite map, 
\(\mathcal{T}_\bX^{} =\mathcal{F}_\bX^{-1}\circ \mathcal{F}_{\bU}^{},\)
where $\mathcal{F}_{\bU}$ is a low-triangular map defined in \eqref{eq:productOfCondMarg}--\eqref{eq:conditional111}. The resulting map $\mathcal{T}_\bX$ is still lower-triangular.
\end{remark}

A bottleneck in constructing the KR rearrangement for high-dimensional random variables is to evaluate the integrals in the marginal densities \eqref{eq:marginal111}. Our strategy is to approximate the density $f_\bX$ using a structured approximation class, which permits the analytical computation of the marginals of the approximate density. This way, the corresponding approximate KR rearrangement can be evaluated explicitly.

\subsection{Approximating the target density}\label{sec:sqrt_approx}

We propose to approximate the target density function $f_\bX$ by a density in the form of
\begin{equation}\label{eq:approx_density}
f_{\hat\bX}(\bx) = \frac{1}{\hat z} \left( \gamma + g(\bx)^2\right) \, \lambda(\bx),\quad \hat z = \gamma + \int_\mathcal{X} g(\bx)^2 \lambda(\bx) \dd \bx .
\end{equation}
Here, $\lambda(\bx)=\lambda(x_1)\hdots\lambda(x_d)$ is a product function of positive weight such that $\int_\R\lambda(x)\d x =1$, $\gamma > 0$ is a scalar and $g\in L^2_\mu(\mathcal{X})$ is a square integrable function. By construction, the density as in \eqref{eq:approx_density} is everywhere positive $f_{\hat\bX}(\bx)>0$.
In practice, we build $f_{\hat\bX}$ as follow.
First we chose $\lambda$ with sufficiently heavy tails such that
\begin{equation}\label{eq:tail}
 \sup_{\bx\in\mathcal{X}} \frac{f_{\bX}(\bx)}{\lambda(\bx)} <\infty ,
\end{equation}
is bounded. In particular, this implies that the ratio of the densities $f_{\bX}(\bx)/f_{\hat\bX}(\bx)$ is bounded from above, which is a sufficient condition to ensures that $f_{\hat\bX}$ can be used as a robust importance sampling density, see \cite[Chapter 9]{Owen_2013}.
Given such a weight function, we denote by $\Phi(\bx)$ the unique (up to an additive constant) potential function such that the target density function writes
\begin{equation}\label{eq:target}
f_\bX(\bx) = \frac{1}{z} \, \exp(-\Phi(\bx))\,\lambda(\bx), \quad z = \int_{\mathcal{X}} \exp(-\Phi(\bx))\,\lambda(\bx) \dd \bx.
\end{equation}
Here, $\bx\mapsto\Phi(\bx)$ can be pointwise evaluated, but the normalizing constant $z$ might not be known in advance.
Next, we build $g$ as an approximation of the square-root of $\exp(-\Phi(\bx))$ by solving the following least-square problem
\begin{equation}\label{eq:approx}
    g = \argmin_{h \in \mathcal{V}_{\mathcal{K}}} \lpnormx{\exp(-{\textstyle\frac12}\Phi) - h}{2} .
\end{equation}
Here, $\mathcal{V}_{\mathcal{K}}\subset L^2_\mu(\mathcal{X})$ is a subspace of functions defined as the span of tensor-product multivariate basis functions, i.e.,
\begin{equation}\label{eq:Vk}
  \mathcal{V}_{\mathcal{K}} = \text{span}\{ \psi_{\bk} \,:\, \bk\in\mathcal{K}  \},
  \quad \psi_{\bk}:\bx\mapsto \prod_{i = 1}^d \psi_{k_i}(x_i),
\end{equation}
where $\mathcal{K} \subset \mathbb{N}_0^d$ is a multi-index set and each $\psi_{k_i}$ is a univariate function indexed by $k_i \in \mathbb{N}_0$.
As explained in the next Section \ref{sec:EvaluatingKR}, the tensor product structure \eqref{eq:Vk} permits to compute analytically any marginals of $f_{\hat\bX}$, which permits evaluating the corresponding KR rearrangement $\mathcal{F}_{\hat \bX}$.
We defer the definition of the functions $\psi_k$ to Section \ref{sec:Psi} and the construction of the index set $\mathcal{K}$ to Section \ref{sec:AdaptiveLS}.
Finally, we $\gamma>0$ to be any positive constant such that $\sqrt{\gamma} \leq \| \exp(-\textstyle\frac12\Phi)-g \|_{L_\lambda^2(\mathcal{X})}$.
The following proposition, which is shown in \cite[Section 3.3]{cui2021deep}, establishes the connection between the error in approximating $\exp(-\frac12\Phi)$ and the error of the random variable $\hat\bX$ as an approximation of the target random variable $\bX$.

\begin{proposition}\label{prop:l2}
    Let $\lambda:\R^d\rightarrow\R_{\geq0}$ be such that $\int_{\R^d}\lambda(\bx) \dd \bx=1$.
    Assume that the approximation $g \in L^2_\lambda(\R^d)$ satisfies $\| \exp(-\textstyle\frac12\Phi)-g \|_{L_\lambda^2(\mathcal{X})} \leq \epsilon$ for some $\epsilon \geq 0$, and the constant $\gamma\geq0$ satisfies
    \(
        \sqrt{\gamma} \leq \epsilon .
    \)
    Then the random variable $\hat\bX$ with density $f_{\hat\bX}$ defined in \eqref{eq:approx_density} satisfies $D_{\rm H}(\nu_{\hat\bX}, \nu_{\bX {\vphantom{\hat\bX}}} ) \leq 2\epsilon / \sqrt z$.
\end{proposition}

Since $\sqrt{z} = \| \exp(-\textstyle\frac12\Phi)\|_{L_\lambda^2(\mathcal{X})}$, for an approximation $g$ with a relative error 
\[
    \frac{\| \exp(-\textstyle\frac12\Phi)-g \|_{L_\lambda^2(\mathcal{X})}}{\| \exp(-\textstyle\frac12\Phi)\|_{L_\lambda^2(\mathcal{X})}} \leq \tau 
\]
for some $\tau \geq 0$ and $\gamma \leq \tau^2 \| \exp(-\textstyle\frac12\Phi)\|_{L_\lambda^2(\mathcal{X})}^2$, Proposition \ref{prop:l2} also gives that the Hellinger distance between the approximate random variable and the target random variable satisfies $D_{\rm H}(\nu_{\hat\bX}, \nu_{\bX {\vphantom{\hat\bX}}} ) \leq 2\tau$.

\subsection{Evaluating the approximate KR rearrangement \texorpdfstring{$\mathcal{F}_{\smash{\hat \bX}}$}{}}\label{sec:EvaluatingKR}
Evaluating the KR rearrangement of the approximate random variable $\hat \bX$ requires marginals of $\hat \bX$. A naive approach computes integrals of $g(\bx)^2$ in the approximate density \eqref{eq:approx_density} over the corresponding variables, which can be computationally inefficient.
In the following proposition, we exploit the product structure of the basis $\psi_{\bk}$ to derive a simple expression for the marginals of $\hat \bX$.
\begin{proposition}\label{prop:marginal}
Let $g \in \mathcal{V}_{\mathcal{K}}$ take the form
\begin{equation}\label{eq:approx_g}
    g(\bx) = \sum_{\bk \in \mathcal{K}} c_{\bk} \psi_{\bk}(\bx),
\end{equation}
for some real-valued coefficients $(c_\bk)_{\bk\in\mathcal{K}}$. For $1\leq i \leq d$, we define the mass matrices $\mM^\lowsup{(i)} \in \R^\lowsup{|\mathcal{K}| \times |\mathcal{K}|}$ as
\begin{equation}\label{eq:Mi}
\mM^{(i)}_{\sigma(\bj) \sigma(\bk)} = \int_{\mathcal{I}} \psi_{j_i}(x_i) \psi_{k_i}(x_i) \lambda(x_i) \dd x_i,
 \quad \bj,\bk \in \mathcal{K} .
\end{equation}
Given a coordinate index set $\mathfrak{q} = (q_1, \ldots, q_s) \subseteq [d]$ with cardinality $s \leq d$, we define the accumulated mass matrix
\(
\mM^\lowsup{(\mathfrak{q})} = \mM^\lowsup{(q_1)}\circ \cdots \circ \mM^\lowsup{(q_s)},
\)
where $\circ$ is the Hadamard product. Considering the rank-revealing factorization $\mM^\lowsup{(\mathfrak{q})}_{\sigma(\bj) \sigma(\bk)} = \sum_{\ell = 1}^{r(\mathfrak{q})} \mL^\lowsup{(\mathfrak{q})}_{\sigma(\bj) \ell}   \mL^\lowsup{(\mathfrak{q})}_{\sigma(\bk) \ell}$, where $r({\mathfrak{q}})$ is the rank of $\mM^\lowsup{({\mathfrak{q}})}$, the density of the marginal random variable $\hat\bX_{\mathfrak{q}^c}$ is
\begin{align}
    f_{{\smash{\hat \bX}}_{\mathfrak{q}^c}}(\bx_{\mathfrak{q}^c}) = \int_{\mathcal{X}_{\mathfrak{q}}} f_{{\smash{\hat \bX}}}(\bx) \dd \bx_{\mathfrak{q}} = \frac{1}{\hat z}\big( \gamma + p(\bx_{\mathfrak{q}^c}) \big) \lambda(\bx_{\mathfrak{q}^c}).
\end{align}
where
\begin{equation}\label{eq:accum_M}
p(\bx_{\mathfrak{q}^c}) = \sum_{\ell = 1}^{r({\mathfrak{q}})}  \Big( \sum_{\bk \in \mathcal{K}}  c_{\bk}  \mL^{({\mathfrak{q}})}_{\sigma(\bk)\ell}  \prod_{i \in {\mathfrak{q}}^c}  \psi_{k_i}(x_i) \Big)^2, \quad \hat z = \gamma + \hspace{-8pt}\sum_{\bj \in \mathcal{K}, \bk \in \mathcal{K}} \hspace{-8pt} c_{\bj} \,c_{\bk} \,\mM^{([d])}_{\sigma(\bj) \sigma(\bk)} .
\end{equation}
\end{proposition}
\begin{proof}
    Since the approximate target density is $f_{\smash{\hat \bX}}(\bx) = \frac{1}{\hat z} \left( \gamma + g(\bx)^2\right) \, \lambda(\bx)$ and the weight function $\lambda(\bx)$ is normalized and of a product form, the density of $\hat\bX_{\mathfrak{q}^c}$ is
    \[
        f_{{\smash{\hat \bX}}_{\mathfrak{q}^c}}(\bx_{\mathfrak{q}^c}) = \frac{1}{\hat z} \int_{\mathcal{X}_{\mathfrak{q}}} \big( \gamma + g(\bx)^2\big) \, \lambda(\bx) \dd \bx_{\mathfrak{q}} = \frac{1}{\hat z} \Big( \gamma + \int_{\mathcal{X}_{\mathfrak{q}}}  g(\bx)^2 \lambda(\bx_\mathfrak{q}) \dd \bx_{\mathfrak{q}} \Big) \lambda(\bx_{\mathfrak{q}^c}).
    \]
    The marginal function $p(\bx_{{\mathfrak{q}}^c}) = \int_{\mathcal{X}_{\mathfrak{q}}}  g(\bx)^2 \lambda(\bx_\mathfrak{q}) \dd \bx_{\mathfrak{q}}$ can be expressed as
    \begin{align*}
    p(\bx_{{\mathfrak{q}}^c}) & = \int_{\mathcal{X}_{\mathfrak{q}}}\Big(\sum_{\bj \in \mathcal{K}} c_{\bj} \prod_{i = 1}^d \psi_{j_i}(x_i) \Big) \Big(\sum_{\bk \in \mathcal{K}} c_{\bk} \prod_{i = 1}^d \psi_{k_i}(x_i) \Big) \prod_{i \in {\mathfrak{q}}} \lambda(x_i)  \dd \bx_{\mathfrak{q}} \\
    & = \sum_{\bj, \bk \in \mathcal{K}} c_{\bj} \, c_{\bk}  \Big( \prod_{i \in {\mathfrak{q}}^c} \psi_{j_i}(x_i)  \psi_{k_i}(x_i) \Big) \Big(\prod_{i \in {\mathfrak{q}}} \int_{\mathcal{I}} \psi_{j_i}(x_i) \psi_{k_i}(x_i) \lambda(x_i) \dd x_i \Big)                                               \\
    & = \sum_{\bj , \bk \in \mathcal{K}} c_{\bj} \, c_{\bk}  \Big( \prod_{i \in {\mathfrak{q}}^c} \psi_{j_i}(x_i)  \psi_{k_i}(x_i) \Big) \;
    \Big( \sum_{\ell = 1}^{r({\mathfrak{q}})} \; \mL^{({\mathfrak{q}})}_{\sigma(\bj)\ell} \mL^{({\mathfrak{q}})}_{\sigma(\bk)\ell} \Big) \\
    & = \sum_{\ell = 1}^{r({\mathfrak{q}})} \Big( \sum_{\bk \in \mathcal{K}}  c_{\bk}  \mL^{({\mathfrak{q}})}_{\sigma(\bk)\ell}  \prod_{i \in {\mathfrak{q}}^c}  \psi_{k_i}(x_i) \Big)^2 .
    \end{align*}
    The constant $\hat z$ is the integration of $(\gamma + g(\bx)^2) \lambda(\bx)$ over all coordinates, and thus its expression can also be obtained using the above formula. This concludes the result.
\end{proof}

\begin{remark}[Variable ordering]
As suggested by Proposition \ref{prop:marginal}, we can obtain marginal density for any subset of variables.
Thus, given an arbitrary variable ordering $\mathfrak{p}=(p_1,\ldots,p_d)$, we can use the above defined procedure to evaluate the KR rearrangement associated with the reordered variable $\hat\bX_{\mathfrak{p}} = (X_{p_1},\ldots,X_{p_d})$. The computational cost is invariant to variable ordering. This is more flexible than the squared-tensor-train methods of \cite{cui2021deep,cui2021conditional}, where marginalizing variables in arbitrary order can significantly increase the computational complexity.

\end{remark}

\section{Adaptive least square approximation in polynomial spaces}\label{sec:adapt_sirt}
The procedure outlined in the previous section can be used with any product-form basis functions. For example, Gaussian process and radial basis functions \cite{buhmann2003radial,wendland2004scattered,williams2006gaussian} with product-form kernels, multivariate wavelet and Fourier bases \cite{daubechies1992ten,de1993construction,devore1988interpolation,dohler2010nonequispaced,gradinaru2007fourier,mallat1989multiresolution}, and sparse spectral polynomials  \cite{shen2010sparse,shen2011spectral,xiu2015stochastic,xiu2002wiener}. Here we consider multivariate basis functions $(\psi_\bk)_{\bk\in \mathbb{N}_0^d}$ obtained by tensorizing orthonormal univariate polynomials $(\psi_k)_{k=0}^\infty$ such that $\langle\psi_j, \psi_k \rangle_{\lambda}= \delta(j,k) $, where $\delta(j,k)$ is the Kronecker delta. This way, the tensor-product basis functions $(\psi_\bk)=(\psi_{k_1})\otimes\cdots\otimes(\psi_{k_d})$ are orthonormal because
\begin{equation}\label{eq:orthognormalBasisPsi}
 \langle\psi_\bj, \psi_\bk \rangle_{\lambda}
 = \prod_{i=1}^d \langle\psi_{j_i}, \psi_{k_i} \rangle_{\lambda}
 = \prod_{i=1}^d  \delta(j_i,k_i)
 = \delta(\bj,\bk) .
\end{equation}

\subsection{Orthogonal polynomial basis}\label{sec:Psi}
We recall a well-known condition on $\lambda$ that ensures polynomials are dense in $L^2_\lambda(\R)$, see \cite{akhiezer1965classical}.
\begin{proposition}
The polynomials are dense in $L^2_\lambda(\R)$ if there exists a constant $a>0$ such that
\begin{equation}\label{eq:light_tails}
  \int_{\R} \exp(a|x|)\lambda(x) \d x< \infty.
\end{equation}
The orthonormal polynomials $(\psi_k)_{k=0}^{\infty}$ with $\mathrm{degree}(\psi_k) = k$ form a Hilbert basis for $L^2_\lambda(\R)$, and therefore any multivariate function $u\in L^2_\lambda(\mathcal{X})$ can be represented as
 $$
  u(\bx) = \sum_{\bk\in\mathbb{N}^d_0} u_\bk \psi_\bk(\bx) ,
  \quad u_\bk = \int u(\bx)\psi_\bk(\bx)\d\lambda(\bx),
  \quad \psi_\bk : \bx \mapsto \prod_{i = 1}^d \psi_{k_i}(x_i).
 $$
 In addition, the sequence $u_\bk$ is square summable, and for any $\mathcal{K}\subset \mathbb{N}^d_0$ we have
 \begin{equation}\label{eq:truncatedSumInL2}
  \lpnormx{u - \sum_{\bk\in\mathcal{K}} u_\bk \psi_\bk  }{2}^2 = \sum_{\bk\notin \mathcal{K}} u_\bk^2 .
 \end{equation}
\end{proposition}

The assumption in \eqref{eq:light_tails} is essentially a condition on the tails of $\lambda$, which is satisfied by a large class of measures. The decomposition \eqref{eq:truncatedSumInL2} is the cornerstone for analyzing approximations based on orthogonal basis. 
In our setting, the target probability density function is necessarily integrable, and thus projecting the square root of the target density onto orthogonal basis yields square summable coefficients.
Given additional regularity assumptions on the target function---typically in the form of holomorphy or mixed smoothness---there exists a sequence of index set $\mathcal{K}_{n}$ with $|\mathcal{K}_n|=n$ such that the truncation error \eqref{eq:truncatedSumInL2} yields a fast (e.g., algebraic) decay. 
We refer to \cite{dung2018hyperbolic,shen2011spectral,shen2010sparse} for details on general polynomial approximations, \cite{IP:SchStu_2012,schillings2016scaling} on applications to inverse problems, and \cite{zech2022sparsea} for analysis of KR rearrangements. In addition, we use \eqref{eq:truncatedSumInL2} to guide the construction of the index set in Section \ref{sec:AdaptiveLS}.
The following defines a well-known class of orthogonal univariate polynomial basis.

\begin{definition}\label{def:jacobi}
Recall the density function of the Beta distribution,
$$\lambda^{\alpha, \beta}(x) = \frac{\Gamma(\alpha+\beta+2)}{\Gamma(\alpha+1)\Gamma(\beta+1)} (1-x)^\alpha(1+x)^\beta,$$
with $\alpha, \beta > -1$, for $x \in \mathcal{I} = (-1, 1)$. The Jacobi polynomials $( J^{\alpha, \beta}_k)_{k = 0}^\infty$ are the eigenfunctions of the Sturm--Liouville problem,
\[
-\frac1{\lambda^{\alpha, \beta}(x)}\frac{\partial }{\partial x}  \Big( \lambda^{\alpha+1, \beta+1}(x) \frac{\partial }{\partial x} J^{\alpha, \beta}_k(x) \Big) = \sigma^{\alpha, \beta}_k J^{\alpha, \beta}_k(x),  
\]
with eigenvalues  $\sigma^{\alpha, \beta}_k = k(k+\alpha+\beta+1)$. %
The normalized polynomials $\psi^{\alpha,\beta}_k(x) = w^{\alpha,\beta}_k J^{\alpha, \beta}_k(x)$, where $w^{\alpha,\beta}_k = \smash{(\int_{-1}^1 J^{\alpha, \beta}_k(x) \lambda^{\alpha, \beta}(x) \dd x )}^{1/2}$, satisfies $\langle \psi^{\alpha,\beta}_j, \psi^{\alpha,\beta}_k \rangle_{\lambda^{\alpha,\beta}} = \delta(j,k)$. 
There are some notable special cases the Jacobi polynomials. With $\alpha = \beta = 0$, we have the Legendre polynomials. With $\alpha = \beta = -\frac12$ and $\alpha = \beta = \frac12$, we have the Chebyshev polynomials of the first kind and the second kind, respectively. 
\end{definition}

The following proposition shows that the mass matrices introduced in Proposition \ref{prop:marginal} admits a simple expression when orthogonal basis functions are used.

\begin{proposition}\label{prop:orth}
Suppose we have orthogonal basis functions $\{\psi_\bk\}_{\smash{\bk\in\mathbb{N}_0^d}}$ as in \eqref{eq:orthognormalBasisPsi}, a multi-index set $\mathcal{K} \subset \mathbb{N}_0^d$ and its corresponding matrix form $\mathsf{K} = \verb|matrix|(\mathcal{K})$. Let $\mathfrak{q}\subseteq {[d]}$ be a coordinate index set. We let $\mathfrak{u}({\mathfrak{q}}) = (u_1(\mathfrak{q}), \ldots, u_m(\mathfrak{q})) \subseteq [|\mathcal{K}|]$ be a set of row indices of $\mathsf{K}_{: \mathfrak{q}}$ such that the submatrix $\mathsf{K}_{\mathfrak{u}({\mathfrak{q}}) \mathfrak{q}}$ contains all unique rows of $\mathsf{K}_{:\mathfrak{q}}$.
Then the matrix $\mM^{(\mathfrak{q})}$ takes the form
\[
    \mM^{(\mathfrak{q})}_{jk} = \delta(\mathsf{K}_{j \mathfrak{q}} ,\mathsf{K}_{k \mathfrak{q}})
    = \prod_{i\in \mathfrak{q}}  \delta(\mathsf{K}_{ji} ,\mathsf{K}_{ki} ) ,
\]
for all $1\leq j,k \leq |\mathcal{K}|$.
In addition, the rank of $\mM^{(\mathfrak{q})}$ is $\mathrm{rank}(\mM^{(\mathfrak{q})}) =m$, the number of unique rows of $\mathsf{K}_{:\mathfrak{q}}$.
Furthermore, $\mM^{(\mathfrak{q})}$ has a decomposition
\begin{equation}\label{eq:orth_decompose} 
    \mM^{(\mathfrak{q})} = \mathsf{P}^{(\mathfrak{q})} (\mathsf{P}^{(\mathfrak{q})} )^\top, \;\; 
    \mathsf{P}^{(\mathfrak{q})}_{\sigma(\bk) \ell} =  \delta(\sigma(\bk) ,u_\ell(\mathfrak{q})), \;\; \bk \in \mathcal{K}, \;\; \ell = 1, \ldots, |\mathfrak{u}({\mathfrak{q}})| .
\end{equation}
Finally with $\mathfrak{q} = [d]$, the matrix $\mM^{([d])}$ is the $|\mathcal{K}|\times |\mathcal{K}|$ identity matrix. 
\end{proposition}

\begin{proof}
Given the coordinate index set $\mathfrak{q}\subseteq {[d]}$, we restrict the set of basis functions $\{\psi_\bk\}_{\bk\in\mathcal{K}}$ to $\bx_\mathfrak{q}$, which leads to a multiset 
\[
    \{\psi_{\bk_\mathfrak{q}}: \bk\in\mathcal{K}\}, \quad \psi_{\bk_\mathfrak{q}}: \bx_\mathfrak{q} \mapsto \prod_{i \in \mathfrak{q}} \psi_{k_i}(x_i),
\]
with cardinality $|\mathcal{K}|$. Elements of $\{\psi_{\bk_\mathfrak{q}} : \bk {\in} \mathcal{K}\}$ may not be unique. With the mapping $\sigma: \mathcal{K} \rightarrow [|\mathcal{K}|]$, we can also express $\{\psi_{\bk_\mathfrak{q}}: \bk{\in}\mathcal{K}\}$ using scalar indices, which yields 
\[ 
    \{\phi_{\sigma(\bk)}: \bk\in\mathcal{K}\}, \quad \phi_{\sigma(\bk)}(\bx_\mathfrak{q}):=\psi_{\bk_\mathfrak{q}}(\bx_\mathfrak{q}).
\]
Given the matrix $\mathsf{K} = \verb|matrix|(\mathcal{K})$, recall that each element of $\{\bk_\mathfrak{q}: \bk\in\mathcal{K}\}$ corresponds to $\sigma(\bk)$-th row of the submatrix $\mathsf{K}_{:\mathfrak{q}}$ and the index set $\mathfrak{u}({\mathfrak{q}})\subseteq [|\mathcal{K}|]$ is constructed such that the submatrix $\mathsf{K}_{\mathfrak{u}({\mathfrak{q}}) \mathfrak{q}}$ contains all unique rows of $\mathsf{K}_{:\mathfrak{q}}$. This way, the set of basis functions 
\(
    \{\phi_{j}: j\in\mathfrak{u}({\mathfrak{q}})\}
\)
contains unique elements of the multiset $\{\phi_{\sigma(\bk)}: \bk\in\mathcal{K}\}$. Each element of the multiset $\{\phi_{\sigma(\bk)}: \bk\in\mathcal{K}\}$ can be expressed as
\[
    \phi_{\sigma(\bk)}(\bx_{\mathfrak{q}}) = \sum_{\ell = 1}^{|\mathfrak{u}({\mathfrak{q}})|} \mathsf{P}^{(\mathfrak{q})}_{\sigma(\bk) \ell }\, \phi_{u_\ell(\mathfrak{q})}(\bx_{\mathfrak{q}}).
\]
where $\mathsf{P}^{(\mathfrak{q})}$ is defined in \eqref{eq:orth_decompose}. By Proposition \ref{prop:marginal} and the above identities, we have
\begin{align*}
    \mM^{(\mathfrak{q})}_{\sigma(\bj) \sigma(\bk)}  & = \int_{\mathcal{X}_{\mathfrak{q}}}  \phi_{\sigma(\bj)}(\bx_{\mathfrak{q}}) \phi_{\sigma(\bk)}(\bx_{\mathfrak{q}}) \lambda(\bx_{\mathfrak{q}}) \d \bx_{\mathfrak{q}} \\
    & = \sum_{\ell = 1}^{|\mathfrak{u}({\mathfrak{q}})|} \sum_{\ell' = 1}^{|\mathfrak{u}({\mathfrak{q}})|} \mathsf{P}^{(\mathfrak{q})}_{\sigma(\bj) \ell } \mathsf{P}^{(\mathfrak{q})}_{\sigma(\bk) \ell' }\, \int_{\mathcal{X}_{\mathfrak{q}}}  \phi_{u_\ell(\mathfrak{q})}(\bx_{\mathfrak{q}})  \phi_{u_{\ell'}(\mathfrak{q})}(\bx_{\mathfrak{q}}) \lambda(\bx_{\mathfrak{q}}) \d \bx_{\mathfrak{q}} \\
    & = \sum_{\ell = 1}^{|\mathfrak{u}({\mathfrak{q}})|} \sum_{\ell' = 1}^{|\mathfrak{u}({\mathfrak{q}})|} \mathsf{P}^{(\mathfrak{q})}_{\sigma(\bj) \ell } \mathsf{P}^{(\mathfrak{q})}_{\sigma(\bk) \ell' }\,\delta(\ell, \ell') = \sum_{\ell = 1}^{|\mathfrak{u}({\mathfrak{q}})|} \mathsf{P}^{(\mathfrak{q})}_{\sigma(\bj) \ell } \mathsf{P}^{(\mathfrak{q})}_{\sigma(\bk) \ell },
\end{align*}
where the second last equality follows from elements of the set $\{\phi_{j}: j\in\mathfrak{u}({\mathfrak{q}})\}$ are normalized and mutually orthogonal with respect to $\lambda$. This concludes the proof.
\end{proof}

Applying \eqref{eq:accum_M} in Proposition \ref{prop:marginal} and the decomposition \eqref{eq:orth_decompose} of Proposition \ref{prop:orth}, the marginal function of $g(\bx)^2$, where $g(\bx) = \sum_{\bk \in \mathcal{K}} c_{\bk} \psi_{\bk}(\bx)$, after integrating over the coordinates indexed by $\mathfrak{q} \subseteq [d]$ is 
\begin{align}\label{eq:orth_marginal}
    p(\bx_{{\mathfrak{q}}^c}) = \int_{\mathcal{X}_{\mathfrak{q}}} g(\bx)^2 \lambda_{\mathfrak{q}}(\bx_{\mathfrak{q}}) \dd \bx_{\mathfrak{q}} = \sum_{\ell = 1}^{ |\mathfrak{u}^{\mathfrak{q}}|} \Big( \sum_{\bk \in \mathcal{K}}  c_{\bk}  \mathsf{P}^{({\mathfrak{q}})}_{\sigma(\bk)\ell}  \prod_{i \in {\mathfrak{q}}^c}  \psi_{k_i}(x_i) \Big)^2.
\end{align}

\begin{algorithm}[h]
\caption{Evaluation of the KR rearrangement. Inputs are the weight function $\lambda$, the system of orthogonal basis functions $(\psi_k)_{k=0}^\infty$, the multi-index set $\mathcal{K}$, the approximation coefficients $(c_\bk)_{\bk\in\mathcal{K}}$, a constant $\gamma$, a variable ordering $\mathfrak{p}=(p_1,\ldots,p_d)$, and a realization $\bu$ of the reference random variable $\bU$. \label{alg:rt}}
\begin{algorithmic}[1]
\Function{EvaluateKR}{$\lambda, (\psi_k)_{k=0}^\infty$, $\mathcal{K}$, $(c_\bk)_{\bk\in\mathcal{K}}$, $\gamma$, $\mathfrak{p}$, $\bu$}
\State $\hat z \leftarrow \gamma + \sum_{\bk \in\mathcal{K}} c_\bk^2$  \Comment{{\scriptsize $\hat z$ is the normalizing constant}}
\State $\bxi \leftarrow \mathcal{R}(\bu)$  \Comment{{\scriptsize $\bxi$ is a realization of the uniform random variable in $[0,1]^d$}}
\For{$t = 1,2,\ldots,d$}
\State  $\mathfrak{q}_{t} \leftarrow  \mathfrak{p}_{[d]\setminus[t]}$ 
\State Find the index set $\mathfrak{u}^{\mathfrak{q}_{t}}$ and compute $\mathsf{P}^{(\mathfrak{q}_t)}$ using \eqref{eq:orth_decompose}
\State Assemble $q(x_t)$ as in \eqref{eq:one_pdf}
\If{$t = 1$} \Comment{{\scriptsize marginal distribution defined by $(\mathsf{P}^{(\mathfrak{q}_1)}, (c_\bk)_{\bk\in\mathcal{K}}, \gamma, \hat z)$}}
\State $x_{p_1}\leftarrow \mathcal{F}^{-1}_{\smash{\hat X}_{p_1} }(\xi_{p_1})$ \vspace{-2pt}
\Else \Comment{{\scriptsize conditional distribution defined by $(\mathsf{P}^{(\mathfrak{q}_{t-1})}, \mathsf{P}^{(\mathfrak{q}_{t})}, (c_\bk)_{\bk\in\mathcal{K}}, \gamma)$}}
\State $x_{p_t}\leftarrow \mathcal{F}^{-1}_{\smash{\hat X}_{p_t} | \smash{\hat \bX}_{\mathfrak{p}_{[t-1]}} }(\xi_{p_t}|\bx_{\mathfrak{p}_{[t-1]}})$ \Comment{{\scriptsize $\bx_{\mathfrak{p}_{[t{-}1]}} {=} (x_{p_1}, ..., x_{p_{t-1}})$}}\vspace{-4pt}
\EndIf
\EndFor
\State \Return $\bx = (x_1, x_2, \ldots, x_d)$ such that $\bx = \mathcal{F}_{\smash{\hat \bX}}^{-1}\circ \mathcal{F}_{\bU}(\bu)$ 
\EndFunction
\end{algorithmic}
\end{algorithm}

Algorithm~\ref{alg:rt} summarizes the procedure for evaluating the KR rearrangement. Note that Steps 2 and 6 of Algorithm~\ref{alg:rt} can be precomputed in an offline phase to reduce the complexity.
Evaluating the inverse distribution functions at Steps 9 and 11 in Algorithm \ref{alg:rt} requires a few iterations of a root finding methods---such as the regula falsi and Newton's methods---to reach an accuracy close to machine precision.
This requires several evaluations of the conditional density $x_t\mapsto f_{\smash{\hat X_{t} | \hat\bX_{[t{-}1]}}}(x_t| \bx_{[t{-}1]})$ for fixed $\bx_{[t{-}1]}$. To realize these operations efficiently, some precomputations can be carried before inverting the conditional distribution functions.
By definition \eqref{eq:approx_density} of $f_{\hat\bX}$, with fixed variables $\bx_{[t{-}1]}$, one can write the conditional density as
\begin{equation}\label{eq:one_pdf}
    f_{\hat X_{t} | \hat\bX_{[t{-}1]}}(x_t| \bx_{[t{-}1]})
     = \frac{1}{\zeta} q(x_{t}) \lambda(x_{t}),
\end{equation}
where
\begin{align*}
 q(x_{t}) = \gamma + \sum_{\ell = 1}^{r} \Big( \sum_{\bk \in \mathcal{K}} \psi_{k_{t}}(x_{t})\mC_{\bk \ell}\Big)^2, \;\; 
 \mC_{\bk \ell} =  c_{\bk}  \mathsf{P}^{([d]\setminus[t])}_{\sigma(\bk)\ell} \prod_{i = 1}^{t-1}  \psi_{k_i}(x_i), \;\;
 \zeta = \gamma + p(\bx_{[t{-}1]}).
\end{align*}
In addition, because $\psi_k$ are polynomials, the function $q(x_t)$ in \eqref{eq:one_pdf} is a polynomial with degree bounded by $2n_t$, where $n_t$ is the maximum degree of the univariate polynomials in coordinate $x_t$, i.e., $n_t= \max\{k_t : \bk\in\mathcal{K}\}$. As a consequence, we have
\begin{equation*}
    f_{\hat X_{t} | \hat\bX_{[t{-}1]}}(x_t| \bx_{[t{-}1]})
     = \bigg( \sum_{j = 0}^{2n_t} a_j \psi_{j}(x_t) \bigg) \lambda(x_t),
\end{equation*}
where the coefficients $a_j=a_j(\bx_{[t{-}1]})$ can be precomputed by applying any exact collocation method in one-dimension with $2n_t$ number of points.
See Appendix \ref{sec:legendre} for an example. Collocation methods can have a less complexity than directly evaluating \eqref{eq:one_pdf} by avoiding the summation of squares over the set $\mathcal{K}$ during root findings. Integrating the conditional density yields the distribution function
\begin{equation}\label{eq:oned_dist}
    \mathcal{F}_{\hat X_{t} | \hat \bX_{[t-1]} }(x_{t}|\bx_{[t-1]})
    = \sum_{j = 0}^{2n_t} a_j  \Big( \int_{-\infty}^{x_t} \psi_{j}(y)  \lambda(y) \dd y \Big).
\end{equation}
In Appendix \ref{sec:polys}, we demonstrate that the integrals in \eqref{eq:oned_dist} can be obtained analytically with a computational complexity linear in $2n_t$ for many commonly used orthogonal polynomials, including Chebyshev polynomials on $(-1,1)$, Legendre polynomials on $(-1,1)$, Laguerre polynomials on $[0,\infty)$, Hermite polynomials on $(-\infty,\infty)$.

\begin{remark}
We denote the maximum cardinality of the univariate basis functions over all coordinates by $n$. The matrix $\mathsf{P}^\lowsup{([d]{\setminus}[t])}$ has $|\mathcal{K}|$ number of non-zero entries as shown in \eqref{eq:orth_decompose}. Thus, in each iteration of Algorithm~\ref{alg:rt}, it costs $O( n |\mathcal{K}| )$ floating point operations (flops) to evaluate the conditional density on collocation points. It costs an additional $O(n\log(n))$ flops to obtain the coefficients  $(a_j)_{j = 1}^{m}$ in \eqref{eq:oned_dist} using fast Fourier transform or similar fast transforms. The overall computational complexity of evaluating the KR rearrangement in Alg.~\ref{alg:rt} is $\mathcal{O}(d n |\mathcal{K}| + dn\log(n))$.
\end{remark}

\begin{remark}
    Although this subsection focuses on orthogonal polynomials, Proposition \ref{prop:orth} and Alg.~\ref{alg:rt} are generally applicable to any orthogonal basis functions. 
\end{remark}

\begin{algorithm}[h]
\caption{Construction of the KR rearrangement using orthogonal basis functions. Inputs are the potential function $\Phi$ of the target density, the weight function $\lambda$, a system of orthogonal basis functions $(\psi_k)_{k=0}^\infty$, and a relative error threshold $\tau$. \label{alg:rt_approx}}
\begin{algorithmic}[1]
\Procedure{ConstructKR}{$\Phi, \lambda, (\psi_k)_{k=0}^\infty, \tau$}
\State Find the multi-index set $\mathcal{K}$ and coefficients $(c_\bk)_{\bk\in\mathcal{K}}$ such that
\[\lpnormx{\sum_{\bk \in \mathcal{K}} c_{\bk} \psi_{\bk}(\bx) - \exp(-\textstyle\frac12\Phi)}{2} \leq \tau \lpnormx{\exp(-\textstyle\frac12\Phi)}{2}. \]
\State Choose a constant $\gamma \leq \tau^2 \|\exp(-\frac12\Phi)\|^2_{L^2\lambda(\mathcal{X})}$
\State \Return $\mathcal{K}$, $(c_\bk)_{\bk\in\mathcal{K}}$ and $\gamma$ that defines the map $\mathcal{F}_{\smash{\hat\bX}}$.
\EndProcedure
\end{algorithmic}
\end{algorithm}

\subsection{Adaptive least square approximation}\label{sec:AdaptiveLS}

We provide the details of constructing the index set $\mathcal{K}$ and the computation of $g\in\mathcal{V}_{\mathcal{K}}$ in \eqref{eq:approx} based on weighted least square methods. Starting from some initial index set $\mathcal{K}_1\subset \N_0^d$, we employ a greedy algorithm  to adaptively expand the current index set $\mathcal{K}_n$ as
$
 \mathcal{K}_{n+1} = \mathcal{K}_n \cup \{ \bk^1_n,\hdots,\bk^{m_n}_n \},
$
where $\bk^1_n,\hdots,\bk^{m_n}_n \in\mathbb{N}^d_0$ are $m_n$ multi-indices to be defined below.
Borrowing ideas from \cite{migliorati2015adaptive,migliorati2019adaptive}, we constrain all sets $\mathcal{K}_n$ in the greedy algorithm to be \emph{downward closed} \cite{chkifa2015breaking,cohen2018multivariate}, in the sense that
\begin{equation}\label{eq:downwardClosedK}
\forall \bk \in \mathcal{K}_n \quad \text{and} \quad \forall \bk' \in \N_0^d, \quad \bk'\leq \bk \quad\implies\quad
    \bk'\in \mathcal{K}_n ,
\end{equation}
where $\bk'\leq \bk$ is a partial order such that $k_i'\leq k_i$ for all $i=1,\hdots,d$. The downward closed property is a key to design a tractable construction of $\mathcal{K}_n$, because it does not allow the index set $\mathcal{K}_n$ to contain any ``holes'', see \cite{cohen2018multivariate}.
The new index set $\mathcal{K}_{n+1}$ remains downward closed if and only if all $\bk^1_n,\hdots,\bk^{m_n}_n$ belong to the \emph{reduced margin} of $\mathcal{K}_n$ defined by
\begin{equation}
 \mathcal{K}_n^\text{RM} = \left\{ \bk\notin \mathcal{K}_n: \bk-\mathbf{e}_i \in \mathcal{K}_n \text{ for all } 1\leq i\leq d \text{ such that } k_i > 0 \right\} ,
\end{equation}
where $\mathbf{e}_i$ denotes the $i$-th canonical vector of $\mathbb{N}^d_0$.
We employ the bulk chasing strategy \cite{migliorati2019adaptive} in order to select the multi-indices from $\mathcal{K}_n^\text{RM}$.
For a given parameter $\theta\in(0,1]$, we select the smallest number $m_n$ of multi-indices $\bk^1_n,\hdots,\bk^{m_n}_n \in \mathcal{K}_n^\text{RM} $ such that
$$
 \sum_{\bk \in \{ \bk^1_n,\hdots,\bk^{m_n}_n \} } e_n(\bk)
 \geq \theta \sum_{\bk \in \mathcal{K}_n^\text{RM} } e_n(\bk) ,
$$
where
\begin{align*}
 e_n(\bk) &= \Big(\int \Big( \exp(-{\textstyle\frac12}\Phi(\bx))-g_n(\bx) \Big)\psi_\bk(\bx) \lambda(\bx) \d \bx \Big)^2 ,\\
 g_n(\bx) &= \min_{h\in\mathcal{V}_{\mathcal{K}_n}} \int \Big( \exp(-{\textstyle\frac12}\Phi(\bx))-h(\bx) \Big)^2 \lambda(\bx)\d \bx .
\end{align*}
When $e_n$ and $g_n$ are known, the above procedure actually selects the candidate index set from the reduced margin that yields the fastest $L^2$ error decay, cf. \eqref{eq:truncatedSumInL2}.
In practice, $e_n$ and $g_n$ are approximated using discrete $L^2$ projection and Monte Carlo samples. We employ the optimal weighted least square method \cite{cohen2017optimal,hampton2015coherence,migliorati2019adaptive,narayan2017christoffel} that introduces an importance sampling scheme to improve the stability of the discrete $L^2$ projection. The importance sampling scheme uses the optimal density
$$
 \Lambda_n(\bx) = \frac{1}{|\mathcal{K}_n|} \sum_{\bk\in\mathcal{K}_n}(\psi_\bk(\bx))^2\lambda(\bx) ,
$$
and likelihood weight $w(\bx) = \lambda(\bx)/\Lambda_n(\bx)$, which yields the following estimators
\begin{align*}
\hat e_n(\bk) &= \Big( \frac{1}{N}\sum_{i=1}^N \big( \exp(-{\textstyle\frac12}\Phi(\bX^{(i)}))-\hat g_n(\bX^{(i)}) \big)\psi_\bk(\bX^{(i)})  w(\bX^{(i)}) \Big)^2 ,\\
\hat g_n(\bx) &= \min_{h\in\mathcal{V}_{\mathcal{K}_n}} \frac{1}{N}\sum_{i=1}^N \big( \exp(-{\textstyle\frac12}\Phi(\bX^{(i)}))-h(\bX^{(i)}) \big)^2 w(\bX^{(i)}),
\end{align*}
where $\bX^{(1)},\hdots,\bX^{(N)}$ are independent copies of $\bX\sim \Lambda_n$. Compare to the classical Monte Carlo method, the optimal weighted least square method ensures the quasi-optimality of the empirical projection $\hat g_n$.

\subsection{Domain transformation for densities supported on \texorpdfstring{$\R^d$}{}}\label{sec:rd}

The Hermite polynomials associated to the Gaussian weight function $\lambda(\bx) \propto \exp(-\frac12 \|\bx\|_2^2)$ provide a natural candidate for approximating densities on the infinite domain $\R^d$; see Appendix \ref{sec:hermite} for details. However, there are some caveats in using Hermite polynomials---or Laguerre polynomials when the support is $(0,\infty)^d$---for approximating random variables. First, inverting the distribution function on a unbounded domain can be numerically challenging towards the tails of the distribution. More importantly, the tail condition in \eqref{eq:tail} may not be satisfied by the weight functions on unbounded domain. In fact, the condition \eqref{eq:light_tails} in Proposition \ref{prop:orth} asserts that the tails of the weight function $\lambda$ need to decay no slower than the exponential function. This does not fullfil the tail condition \eqref{eq:tail} for many heavy-tail random variables. Here we discuss different options we have in order to circumvent this issue.

A natural option is to apply domain truncation based on the \emph{a priori} decay rate of the target density towards the tail. Doing this, however, introduce an additional approximation step which might degrade the accuracy of the method. 

The domain transformation method \cite{boyd2001chebyshev,shen2014approximations} is an alterative to handle densities on $\R^d$ while satisfying the tail condition  \eqref{eq:tail}. For brevity, we define domain mappings for univariate densities defined on $\R$. By taking tensor product of the one-dimensional weight function, the method can be generalized to multi-dimensional settings. As a starting point, one maps the real line $\R$ to the interval $(-1,1)$ using a diffeomorphism $z:\R\rightarrow(-1,1)$. Given a reference weight function $\lambda^\text{ref}$ on $(-1,1)$ and associated orthonormal basis functions
 $(\psi_i)_{i = 0}^{\infty}$, the extended basis functions 
\begin{equation}\label{eq:extended_basis}
    (\phi_i)_{i = 0}^{\infty}, \quad \phi_i(x) = \psi_i(z(x)) ,
\end{equation}
forms a set of orthonormal basis functions for $L_{\mu}^2(\R)$ with respect to the weight 
\begin{equation}\label{eq:extended_weight}
   \lambda(x) = \lambda^\text{ref}(z(x)) z'(x).
\end{equation}
This can be shown by a change of variable $x=x(z)$,
\begin{align*}
    \delta(j,k) = \int_{-1}^{1} \psi_j(z) \psi_k(z) \lambda^\text{ref}(z)  \dd z =
    \int_{-\infty}^\infty \phi_j(x)\phi_k(x) \lambda(x) \dd x .
\end{align*}
Table \ref{table:mappings} provides some commonly used domain mappings and the induced weight function $\lambda(x)$ based on $\lambda^{\rm ref}(z) = \frac{1}{\pi}(1-z^2)^{-1/2}$, which is the weight function of the Chebyshev polynomials of the first kind.

\begin{table}[h]
\footnotesize
\caption{Examples of domain mappings, inverse mappings and the induced weight function $\lambda(x)$ for the weight function $\lambda^\text{ref}(z)=\frac{1}{\pi}(1-z^2)^{-1/2}$ on $(-1,1)$.}\vspace{-6pt}\label{table:mappings}
\centering
\begin{tabular}{l|ll|ll|l}
    \hline
    & $z(x)$ & $ z'(x)$ &  $x(z)$ & $x'(z)$ & $\lambda(x)$ $\vphantom{\big)}$\\ \hline
    logarithmic  & $\tanh(x)$ & $1 - \tanh(x)^2$ & $\displaystyle\frac12 \log\Big(\frac{1+z}{1-z}\Big)$ & $\displaystyle\frac1{1 - z^2}$ & $\frac{\displaystyle \mathrm{sech}(x)}{\displaystyle\pi} \vphantom{\frac{\big)}{\big)}}$ \\\hline
    algebraic  & $\displaystyle\frac{x}{\sqrt{1+x^2}}$ & $\displaystyle\frac1{(1 + x^2)^{3/2}}$ & $\displaystyle\frac{z}{\sqrt{1-z^2}}$ & $\displaystyle\frac1{(1 - z^2)^{3/2}}$ &  $\displaystyle\frac{1}{\pi(1+x^2)}\vphantom{\frac{\big)}{\big)}}$ \\\hline
\end{tabular}
\end{table}

We observed that the tail condition \eqref{eq:tail} can be resolved by the domain transformation. For example, using the algebraic mapping in Table \ref{table:mappings}, the ratio between the target density and the weight function satisfies 
\[
\log \Big(\frac{f_X(x)}{\lambda(x)}\Big) = \mathcal{O} \big( \log(1+x^2) + \log f_X(x) \big) = \mathcal{O}(\log|x| + \log f_X(x)).
\]
Thus, for any sufficiently smooth target density with a log-density function decays faster than the logarithmic rate, we have the tail condition $\sup_{x\in\R} f_X(x) / \lambda(x) < \infty$. In fact, the Cauchy distribution, in which the mean and the variance are not defined, presents a boundary case in this situation.

Using the mapped basis functions $(\phi_j)_{j=0}^{\infty}$ on $\R$ and the induced weight $\lambda(x)$, one may directly apply the square-root approximation procedure to the target density $f_X(x)$. More interestingly, as shown in Appendix \ref{sec:domain_mapping}, we can equivalently pull back the target density $f_X(x)$ using the inverse mapping $x(z)$, and then apply the square-root approximation procedure to the pullback density $f_Z(z) = f_X(x(z))x'(z)$ using the reference weight $\lambda^{\rm ref}(z)$ and basis functions $(\psi_j)_{j=0}^\infty$ defined on $(-1,1)$. This allows us to recycle the well-defined implementation procedures for bounded domains, while the convergence analysis based on the mapped basis functions $(\phi_j)_{j=0}^{\infty}$ on $\R$, see \cite{shen2014approximations} and references therein, still applies. 

\section{Self-reinforced approximations}\label{sec:dirt}

For target densities concentrated to a small subdomain or have complicated correlation structures, it may require a high-order index set to construct in one step an approximation to the target density $f_\bX$. 
In previous works \cite{chen2017hessian,schillings2016scaling}, it has shown that appropriate linear preconditioning/scaling can significantly reduce the complexity of sparse polynomial approximations. Here we present the general nonlinear preconditioning framework for approximating concentrated probability densities using orthogonal basis functions. Our approach is \emph{self-reinforced} as it uses the KR rearrangements defined by existing function approximations to precondition the next function approximation that targets a new density with increasing complexity. We also provide some concrete implementation strategies. 

\subsection{Layered construction}
Given a target random variable $\bX$ with density $f_\bX$, we ultimately build a composition of transport maps 
$$
 \mathcal{T}_L = \mathcal{Q}_1 \circ \mathcal{Q}_2 \circ \hdots\circ \mathcal{Q}_L,
$$
guided by a sequence of bridging densities
\[
f_{\bX_{1}} ,\, f_{\bX_{2}} ,\,\hdots,\, f_{\bX_{L}} := f_\bX.
\]
At the end of the construction, the pushforward of a reference random variable $\bU$ under the composite map $\mathcal{T}_L$  approximates $\bX$.

The layers $(\mathcal{Q}_\ell)_{\ell = 1}^L$ in the composite map are built recursively. At level $\ell$, the composed map $\mathcal{T}_{\ell}=\mathcal{Q}_1 \circ \hdots\circ \mathcal{Q}_\ell$ approximates the bridging density $f_{\bX_\ell}$ by $f_{\hat\bX_\ell} := (\mathcal{T}_{\ell})_\sharp f_{\bU}$. In the new iteration, the current map $\mathcal{T}_{\ell}$ preconditions the next bridging density as the pullback 
\(
    \mathcal{T}_{\ell} ^\lowsup{\sharp} f_{\bX_{\ell{+}1}}.
\)
We construct a new polynomial approximation to (the unnormalized version of) the function $\sqrt{\smash{\mathcal{T}^\lowsup{\sharp}_{\ell}} f^{}_{\smash{\bX_{\ell{+}1}}}}$, and then derive the corresponding KR rearrangement $\mathcal{Q}_{\ell+1}$ such that $(\mathcal{Q}_{\ell+1})_\sharp f_\bU$ approximates $\smash{\mathcal{T}_{\ell}^\lowsup{\sharp} f_{\bX_{\ell{+}1}}}$ as detailed in Section \ref{sec:sirt}. This updates the composite map as $\mathcal{T}_{\ell+1} = \mathcal{T}_{\ell}\circ \mathcal{Q}_{\ell+1}$, which approximates $f_{\bX_{\ell{+}1}}$ by $f_{\hat\bX_{\ell+1}} := (\mathcal{T}_{\ell+1})_\sharp f_{\bU}$. Under mild assumptions, the following proposition shows the convergence of the self-reinforced construction process.

\begin{proposition}\label{prop:recur}
Assume there exist a constant $\eta \in (0, 1)$ such that the bridging densities satisfy $\sup_{1\leq \ell < L} D_\mathrm{H}\big( f_{\smash{\bX_{\ell}}}, f_{\bX_{\ell+1}} \big) \leq \eta$. Assume each map $\mathcal{Q}_\ell$ has an error reduction property, 
\begin{equation}\label{eq:err_redu}
    D_\mathrm{H}\big( (\mathcal{Q}_{\ell+1} )_\sharp f_\bU, \mathcal{T}_{\ell}^\sharp f_{\bX_{\ell+1}} \big) \leq  \omega D_\mathrm{H}\big( f_\bU, \mathcal{T}_{\ell}^\sharp f_{\bX_{\ell+1}} \big)    
\end{equation}
for some constant $\omega \in [0,1)$. Let the initial map  $\mathcal{Q}_1$ satisfy $D_\mathrm{H}( (\mathcal{Q}_{1} )_\sharp f_\bU, f_{\bX_{1}} ) \leq \omega \eta$. Then, the resulting composite map $\mathcal{T}_L$ satisfies 
\[
    D_\mathrm{H}\big((\mathcal{T}_{L} )_\sharp f_\bU, f_{\bX_{L}} \big) \leq \frac{\omega}{1-\omega} \eta.
\]
\end{proposition}

\begin{proof}
Because the Hellinger distance satisfies the triangle inequality, we have
\begin{align*}
 D_\mathrm{H}\big( f_{\smash{\hat\bX_{\ell+1}}}, f_{\bX_{\ell+1}} \big) 
 &=D_\mathrm{H}\big( (\mathcal{T}_{\ell}\circ \mathcal{Q}_{\ell+1} )_\sharp f_\bU, f_{\bX_{\ell+1}} \big) \\
 &=D_\mathrm{H}\big( (\mathcal{Q}_{\ell+1} )_\sharp f_\bU, \mathcal{T}_{\ell}^\sharp f_{\bX_{\ell+1}} \big)\\
 & \leq  \omega D_\mathrm{H}\big( f_\bU, \mathcal{T}_{\ell}^\sharp f_{\bX_{\ell+1}} \big)\\
 & =  \omega D_\mathrm{H}\big( f_{\smash{\hat\bX_{\ell}}}, f_{\bX_{\ell+1}} \big)\\
 &\leq \omega \Big(  D_\mathrm{H}\big( f_{\smash{\hat\bX_{\ell}}}, f_{\bX_{\ell}} \big) +  D_\mathrm{H}\big( f_{\smash{\bX_{\ell}}}, f_{\bX_{\ell+1}} \big) \Big)\\
 &\leq \omega \Big(  D_\mathrm{H}\big( f_{\smash{\hat\bX_{\ell}}}, f_{\bX_{\ell}} \big) + \eta \Big),
\end{align*}
for any $0\leq\ell<L$. With $D_\mathrm{H}\big( f_{\smash{\hat\bX_{1}}}, f_{\bX_{1}} \big) \leq \omega \eta$, we have 
\begin{align*}
 D_\mathrm{H}\big( f_{\smash{\hat\bX_{L}}}, f_{\bX_{L}} \big) 
 &\leq \omega^{L} \eta + \omega^{L-1}\eta+\hdots+  \omega \eta \leq \frac{\omega}{1-\omega} \eta,
\end{align*}
which concludes the proof. 
\end{proof}

The error reduction property in \eqref{eq:err_redu} is a natural assumption, as it essentially states that we need to build a new map $\mathcal{Q}_{\ell+1}$ that is better than an identity map by a factor $\omega < 1$. We also highlight that there are other ways of enriching the approximation power of sparse polynomials by local approximations, see \cite{conrad2016accelerating,yan2019adaptive} and references therein. The polynomial-based KR rearrangement presented here can be potentially integrated with these local approximations to further improve the accuracy.

\subsection{Bridging densities and adaptation}\label{sec:bridging}
Here we discuss the construction of the bridging densities in the context of Bayesian inverse problems. We denote the observed data by $\by \in \R^m$, the forward model by $G:\mathcal{X} \rightarrow \R^m$, and assume the observation noises are independently and identically distributed (i.i.d.) zero mean Gaussian random variables with variance $\sigma_n^2$. The likelihood function takes the form
\[
\mathcal{L}^{y}(\bx) \propto \exp(-\Phi^y(\bx)), \quad \Phi^y(\bx) = \frac1{2\sigma_n^2}\|\by - G(\bx)\|^2, 
\]
and we express the prior density as
\(
\pi_0(\bx) \propto \exp(-\Phi_0(\bx))\,\lambda(\bx).
\)
This way, the target posterior density, which is proportional to $\mathcal{L}^{y}(\bx) \pi_0(\bx)$, can be expressed as 
\[
f_\bX(\bx) = \frac{1}{z} \, \exp(-\Phi^y(\bx) - \Phi_0(\bx))\,\lambda(\bx), \quad z = \int_{\mathcal{X}} \exp(-\Phi^y(\bx) - \Phi_0(\bx))\,\lambda(\bx) \dd \bx.
\]
The following gives some examples of the bridging density.

\begin{example}{Tempering.} \label{ex:bridge1} 
Following the works of \cite{gelman1998simulating}, one may use a sequence of temperatures $0 < \beta_1 \leq \beta_1\leq \cdots\leq \beta_L=1$ to construct bridging densities 
\[
    f_{\bX_{\ell}} = \frac{1}{z_\ell} \exp(-\beta_\ell \Phi^y(\bx) - \Phi_0(\bx))\,\lambda(\bx), \quad z_\ell = \int \exp(-\beta_\ell \Phi^y(\bx) - \Phi_0(\bx))\,\lambda(\bx) \d \bx,
\]
for $\ell = 1, \ldots, L$ by tempering the likelihood function. For Bayesian inverse problems, a temperature $\beta_\ell < 1$ corresponds to a likelihood function with increasing variance. 
\end{example}

\begin{example}{Data batching.} \label{ex:bridge2} 
Consider a data set of with cardinality $m$, one can construct $L$ number of disjoint index sets $\mathcal{J}_1, \ldots, \mathcal{J}_L$ such that $\mathcal{J}_i \cap \mathcal{J}_j = \emptyset$ for all $i \neq j$, $i,j \in [L]$  and $\cup_{i = 1}^L \mathcal{J}_i = [m]$. For each $\mathcal{J}_i$, we can define a data-misfit function 
\[
    \Phi^y(\bx; \mathcal{J}_i) = \frac1{2\sigma_n^2}\sum_{j \in \mathcal{J}_i} (y_j - G_j(\bx) )^2, 
\]
where $G_j:\mathcal{X}\rightarrow \R$ is the forward model for $i$-th observable. This way, a sequence of bridging densities can be defined as
\[
    f_{\bX_{\ell}} = \frac{1}{z_\ell} \exp\Big(- \sum_{i = 1}^\ell \Phi^y(\bx; \mathcal{J}_i) - \Phi_0(\bx)\Big)\,\lambda(\bx), \quad \ell = 1, \ldots, L,
\]
where
\(
    z_\ell = \int \exp(- \sum_{i = 1}^\ell \Phi^y(\bx; \mathcal{J}_i) - \Phi_0(\bx))\,\lambda(\bx)\d \bx.    
\)
This bridging density is particular useful when the computational cost of simulating the forward model grows proportionally to the cardinality of observables. 
\end{example}

Bridging densities in above examples can be generalized to problems with non-Gaussian observation noises and other non-Bayesian inference problems. Note that those two types of bridging densities can be used together. For example, tempering can be used within each data batch. 
As suggested by Proposition \ref{prop:recur}, controlling the Hellinger distance between adjacent bridging densities is important in controlling the overall Hellinger error of the composite map. This can be achieved by adaptively choose either the next temperature or the next data batch in our examples. 

Using Example \ref{ex:bridge1}, we discuss the adaptive tempering method introduced in the sequential Monte Carlo literature, e.g., \cite{beskos2016convergence,kantas2014sequential,myers2021sequential}, to provide a precise control of the Hellinger distance in our bridging densities. We want to control the Hellinger distance between bridging densities, so that it satisfies
\(
D_\mathrm{H}\big( f_{\bX_\ell} ,   f_{\bX_{\ell+1}} \big) \approx \eta
\)
for $\eta \in (0, 1)$.
Defining $\Delta = \beta_{\ell+1} - \beta_\ell$, the squared Hellinger distance can be expressed as
\begin{align}
D_\mathrm{H}(f_{\bX_\ell} ,   f_{\bX_{\ell+1}})^2 & = \frac12 \int \Big(\sqrt{\frac{f_{\bX_{\ell+1}}}{f_{\bX_{\ell}}}}- 1\Big)^2 f_{\bX_\ell} \dd \bx  = 1 - \frac{I}{\sqrt{z_{\ell} z_{\ell{+}1}} }  . \label{eq:hellinger_comput}
\end{align}
where
\[
I = \int \exp\Big( - \big(\frac{\Delta}{2}+\beta_\ell\big)\Phi^y(\bx) - \Phi_0(\bx)\Big) \lambda(\bx) \dd \bx     .
\]
We can utilize the current composition of transport maps $\mathcal{T}_\ell$ to estimate the squared Hellinger distance for selecting the temperature increment $\Delta$. 
The pushforward density $f_{\hat\bX_\ell}= (\mathcal{T}_\ell)_\sharp f_\bU $ defines an importance sampling formula for estimating the terms
\begin{align}
z_{\ell{+}1} & = \int \frac{\exp( - (\Delta+\beta_\ell)\Phi^y(\bx) - \Phi_0(\bx))\lambda(\bx)}{f_{\hat\bX_\ell}(\bx)} f_{\hat\bX_\ell}(\bx) \dd \bx ,\label{eq:hellinger_comput1}  \\z_{\ell} & = \int \frac{\exp( - \beta_\ell \Phi^y(\bx) - \Phi_0(\bx))\lambda(\bx)}{f_{\hat\bX_\ell}(\bx)} f_{\hat\bX_\ell}(\bx) \dd \bx,\label{eq:hellinger_comput2}  \\
I & = \int \frac{\exp( - (\Delta/2+\beta_\ell)\Phi^y(\bx) - \Phi_0(\bx))\lambda(\bx)}{f_{\hat\bX_\ell}(\bx)} f_{\hat\bX_\ell}(\bx) \dd \bx, \label{eq:hellinger_comput3}
\end{align}
in \eqref{eq:hellinger_comput}.
One can draw samples $(X_\ell^i)_{i = 1}^N \sim f_{\smash{\hat\bX_\ell}}$ and compute
\begin{align}
F^i_\ell & = \Phi^y(\bX_\ell^i), \\
K^i_\ell & = \beta_\ell \Phi^y(\bX_\ell^i) + \Phi_0(\bX_\ell^i) - \log \lambda(X_\ell^i) + \log  f_{\hat\bX_\ell}(X_\ell^i)\label{eq:sample_K}
\end{align}
Then, the formulas in \eqref{eq:hellinger_comput}--\eqref{eq:hellinger_comput3} lead to the estimated squared Hellinger distance
\begin{align}
    D_{\ell,+}(\Delta) := 1 -  \frac{\sum_{i = 1}^N \exp( - \frac\Delta2 F_\ell^i - K_\ell^i )}{\sqrt{ \sum_{i = 1}^N \exp(- K_\ell^i )} \sqrt{\sum_{i = 1}^N \exp( -\Delta F_\ell^i - K_\ell^i ) } }. \label{eq:hellinger_est1}
\end{align}
which is  non-negative (with $\Delta \geq 0$) by the Cauchy-Schwartz inequality. 

To numerically satisfy the error reduction property in Proposition \ref{prop:recur}, we also want to estimate the (squared) Hellinger distance between the bridging density $f_{\bX_\ell}$ and its approximation $f_{\smash{\hat\bX_\ell}}$. Following a similar derivation as above, we have
\begin{align}
    D_\mathrm{H}(f_{\smash{\hat\bX_\ell}} ,   f_{\bX_{\ell}})^2   \approx D_\ell := 1 - \frac{\sum_{i = 1}^N \exp(-\frac12 K_\ell^i ) }{\sqrt{ \sum_{i = 1}^N \exp(- K_\ell^i )}}, \label{eq:hellinger_est2}
\end{align}
where $K_\ell^i$ is defined in \eqref{eq:sample_K}. The estimator $D_\ell$ is non-negative by Jensen's inequality. The estimators $D_{\ell,+}(\Delta)$ and $D_{\ell}$ are shift-invariant to $K^i_\ell$.

\begin{algorithm}[h]
\caption{Self-reinforced KR rearrangement with adaptive tempering.}
\label{alg:dirt_adapt}
\begin{algorithmic}[1]
\Procedure{LayeredKR}{$\Phi^y$, $\Phi_0$, $\lambda$, $(\psi_k)_{k=0}^\infty$, $\beta_1$, $\omega$, $\eta$} 
\State $\ell \leftarrow 1$
\State $\mathcal{F}_{\smash{\hat\bX}_\ell} \leftarrow$ \Call{ConstructKR}{$\Phi', \lambda, (\psi_k)_{k=0}^\infty, \tau$} \Comment{{\scriptsize $\Phi' = \beta_\ell \Phi^y + \Phi_0$, $\tau=\omega\eta$}}
\State $\mathcal{T}_{\ell} \leftarrow \mathcal{F}_{\smash{\hat\bX}_\ell}^{-1}\circ \mathcal{F}_\bU$  \Comment{{\scriptsize $\mathcal{Q}_{1}\equiv\mathcal{T}_{1}$}}

\While{$\beta_\ell < 1$}
\State $\ell \gets \ell {+} 1$
\State $(\beta_{\ell}, \varepsilon_{\ell-1}) \gets$ \Call{NextBeta}{$\Phi^y$, $\Phi_0, \beta_{\ell-1},\mathcal{T}_{\ell-1}, \eta$} 
\State $\mathcal{F}_{\smash{\hat\bX}_\ell} \gets$ \Call{ConstructKR}{$\Phi', \lambda, (\psi_k)_{k=0}^\infty, \tau$} \Comment{{\scriptsize $\Phi' {=} \log \lambda {-} \log \mathcal{T}_\ell^\sharp f_{\bX_{\ell}}$, $\tau {=} \omega \varepsilon_{\ell{-}1}$}}
\State $\mathcal{Q}_{\ell}\gets \mathcal{F}_{\smash{\hat\bX}_\ell}^{-1}\circ\mathcal{F}_\bU$, $\mathcal{T}_{\ell} \gets \mathcal{T}_{\ell-1}\circ \mathcal{Q}_{\ell}$
\EndWhile
\State $L\gets\ell$ and \Return $\mathcal{T}_L$
\EndProcedure
\end{algorithmic}
\begin{algorithmic}[1]
\Procedure{NextBeta}{$\Phi^y$, $\Phi_0, \beta_\ell, \mathcal{T}, \eta$}
    \State Draw samples $\bX_\ell^i = \mathcal{T}(\bU_i)$ where $\bU_i\sim f_\bU$ for $i = 1, \ldots, N$. \vspace{1pt}
    \State $\varepsilon_\ell \gets \sqrt{D_{\ell}}$ where $D_{\ell}$ is given in \eqref{eq:hellinger_est2}.\vspace{1pt} \Comment{{\scriptsize Error estimate}}
    \State Find $\Delta$ such that $D_{\ell,+}(\Delta) = \eta^2$, where $D_{\ell,+}(\Delta)$ is given in \eqref{eq:hellinger_est1}.\vspace{1pt}
    \State \Return $\beta' = \min(\beta + \Delta, 1)$ and $\varepsilon_\ell$.\vspace{2pt}
\EndProcedure
\end{algorithmic}
\end{algorithm}


\section{Numerical examples}\label{sec:numerics}

We demonstrate the efficiency of our proposed methods on three inverse problems governed by ODEs and PDEs. We compare the self-reinforced KR rearrangements built using sparse polynomials (cf. Section \ref{sec:adapt_sirt}) and the tensor-train (TT) construction of \cite{cui2021deep}, compare the self-reinforced and single-layered constructions, and test the performance of various bridging densities. We also offer some concluding remarks at the end to discuss the numerical results. 
Note that the weighted least square approximation and the adaptation used in Alg.~\ref{alg:dirt_adapt} are subject to randomized estimates. Thus, we compute 9 repeated experiments to report the means and standard deviations of performance results in all examples.

\subsection{Susceptible-Infected-Removed model}

We first apply our methods to estimate the parameters of a compartmental susceptible-infectious-removed (CSIR) model. The CSIR model considered here is a simplified version of the model in \cite{DGKP-SEIR-2021}. We first divide a geographical area  into $K\in\mathbb{N}$ compartments, and denote the numbers of susceptible, infectious and removed
individuals in the $k$th compartment at a given time $t$ by $S_k(t)$,
$I_k(t)$ and $R_k(t)$, respectively. Then, the interaction among the individuals within each compartment and across different compartments is modelled by the following system of differential equations
\begin{align}
\left\{
\begin{array}{ll}
    \dot{S}_k &  =  -\theta_k S_k I_k  + {\frac{1}{2}}\sum_{j \in \mathcal{J}_k} (S_{j} - S_k),\\
    \dot{I}_k & =  \theta_k S_k I_k - \nu_k I_k  + {\frac{1}{2}}\sum_{j \in \mathcal{J}_k} (I_{j}  - I_k), \\
    \dot{R}_k & =  \nu_k I_k +{\frac{1}{2}}\sum_{j \in \mathcal{J}_k} (R_{j} - R_k),
\end{array}
\right. \label{eq:sir}
\end{align}
where $\mathcal{J}_k$ is the index set containing all neighbours of the $k$th compartment.
The behaviour of the above system of differential equations is governed by parameters $\theta_k\in \mathbb{R}$ and $\nu_k \in \mathbb{R}$ for $k = 1, \ldots, K$, which represent the infection and recovery rate in the $k$th compartment, respectively.

The differential equations in \eqref{eq:sir} are solved for the time interval $t \in [0,5]$ with fixed inhomogeneous initial states \( S_k(0) = 99 - K + k, \) \(I_k(0) = K+1 - k,\) and \( R_k(0) = 0 \) for \(k=1,\ldots,K\). We aim to estimate the unknown parameters $\bx =
(\theta_1,\nu_1, \ldots , \theta_K,\nu_K) \in \R^{2K}$ from noisy observations of $I_k(t), k=1,\ldots,K,$ at $6$ equidistant time points. We specify a uniform prior on the domain $[0,2]$ for each of $\theta_k$ and $\nu_k$, which leads to $\pi_0(\bx) = \prod_{k=1}^{2K}\indi_{[0,2]}(x_k)$. In this experiment, we use a``true'' parameter
\(
\bx_\text{true} = [0.1, 1, \ldots, 0.1, 1],
\)
to simulate synthetic observations
\begin{equation*}
  \by_{k,j} = I_k\Big(\frac{5j}{6} ; \bx_\text{true}\Big) + \eta_{k,j}, \quad k=1,\ldots,K, \quad j=1,\ldots,6,
\end{equation*}
where $\eta_{k,j}$ is a zero-mean standard Gaussian measurement noise $\mathcal{N}(0,1)$. This gives the likelihood function 
\begin{equation}
  \mathcal{L}^{y}(\bx) \propto \exp(-\Phi^y(\bx)), \quad \Phi^y(\bx)= \frac12 \sum_{k=1}^{K} \sum_{j=1}^{6} \Big\{I_k\Big(\frac{5j}{6}; \bx\Big) - \by_{k,j}\Big\}^2.
\end{equation}

We consider a compartment model defined on a one-dimensional lattice, in which the $k$th compartment is only connected to compartments with adjacent indices $k-1$ and $k+1$. We impose periodic boundary conditions such that $Z_{K+1} = Z_{1}$ and $Z_{0} = Z_{K}$ for $Z \in \{S, I, R\}$. By varying the number of compartments, and hence the parameter dimension, we can test the scalability of proposed methods. The differential equations in \eqref{eq:sir} are solved for the time interval $t \in [0,5]$ with fixed inhomogeneous initial states \( S_k(0) = 99 - K + k, \) \(I_k(0) = K+1 - k,\) and \( R_k(0) = 0 \) for \(k=1,\ldots,K\).
The differential equations are solved by the explicit Runge--Kutta method with adaptive time steps that control both absolute and relative errors to be within $10^{-6}$. 

\paragraph{Comparison of single-layered and self-reinforced constructions} We demonstrate the benefit of using the multi-layered construction on the CSIR model with a single compartment, i.e., $K = 1$. We set up the self-reinforced KR rearrangement using adaptive tempering, Legendre polynomials with a maximum order of 30 on each coordinate, and a stopping tolerance of $0.05$ for the least square approximation. During the construction, Alg.~\ref{alg:dirt_adapt} generates three layers of sparse polynomial approximations with cardinalities $\{293 \pm 55, 268 \pm 18, 189\pm 2\}$. It takes a total of $2420 \pm 348$ density evaluations to construct the composite map. The resulting Hellinger error is $0.0181 \pm 0.0046$. To set up the single-layered KR rearrangement, we use Legendre polynomials with a maximum order of 60 on each coordinate. In this example, by using all polynomial basis with cardinality $3721$ (which takes $14884$ density evaluations in the weighted least square procedure), the resulting KR rearrangement yields a Hellinger error $0.375\pm 0.023$. Compared to the single-layered construction, the self-reinforced construction has an error of about 5 percent but only uses 20 percent of the total number of basis functions. It is worth mentioning that we are not able to apply the single-layered KR rearrangement for problems with larger numbers of compartments. Figure~\ref{fig:csir_single} shows the comparison of the approximate posterior density constructed by the self-reinforced and single-layer constructions. 

\begin{figure}[h]
\begin{tikzpicture}
\begin{axis}[%
width=0.36\linewidth,
height=0.36\linewidth,
title style={at={(0.5,0.85)}},
x label style={at={(0.5,-0.05)}},
y label style={at={(-0.1,0.5)},rotate=0},
xmin=0.0,xmax=2,ymin=0.35,ymax=1.8,
xtick={0.5,1.5},
ytick={0.7,1.4},
xlabel={$\theta_1$},ylabel={$\nu_1$},title={$f_{\bX_1}$},
]
\pgfplotsset{
    contour/every contour plot/.style={labels=false},
}
\addplot[contour prepared,contour prepared format=matlab,contour/draw color={black},opacity=0.3,line width=2pt] table{bridge_1.dat};
\addplot[contour prepared,contour prepared format=matlab,line width=0.5pt] table{bridge_approx_1.dat}; 
\end{axis}
\end{tikzpicture}
\hspace{-1em}
\begin{tikzpicture}
\begin{axis}[%
width=0.36\linewidth,
height=0.36\linewidth,
title style={at={(0.5,0.85)}},
x label style={at={(0.5,-0.05)}},
y label style={at={(-0.1,0.5)},rotate=0},
xmin=0.05,xmax=0.32,ymin=0.65,ymax=1.28,
xtick={0.1,0.25},
ytick={0.8,1.1},
xlabel={$\theta_1$},ylabel={$\nu_1$},title={$f_{\bX_2}$},
]
\pgfplotsset{
    contour/every contour plot/.style={labels=false},
}
\addplot[contour prepared,contour prepared format=matlab,contour/draw color={black},opacity=0.3,line width=2pt] table{bridge_2.dat};
\addplot[contour prepared,contour prepared format=matlab,line width=0.5pt] table{bridge_approx_2.dat}; 
\end{axis}
\end{tikzpicture}
\hspace{-1em}
\begin{tikzpicture}
\begin{axis}[%
width=0.36\linewidth,
height=0.36\linewidth,
title style={at={(0.5,0.85)}},
x label style={at={(0.5,-0.05)}},
y label style={at={(-0.1,0.5)},rotate=0},
xmin=0.085,xmax=0.145,ymin=0.85,ymax=1.06,
ytick={0.9,1},
xtick={0.1,0.13},
xticklabels={0.1,0.13},
xlabel={$\theta_1$},ylabel={$\nu_1$},title={$f_{\bX_3}=f_{\bX}$},
legend style={at={(0.55,1)},anchor=north west,fill=none},
legend cell align={left},
legend image post style={
    sharp plot,
    draw=\pgfkeysvalueof{/pgfplots/contour/draw color},
},
legend entries={exact,self-reinforced,single-layered}
]
\pgfplotsset{
    contour/every contour plot/.style={labels=false},
}
\addplot[contour prepared,contour prepared format=matlab,contour/draw color={black},opacity=0.3,line width=2pt] table{bridge_3.dat};
\addplot[contour prepared,contour prepared format=matlab,line width=0.5pt] table{bridge_approx_3.dat}; 
\addplot[contour prepared,contour prepared format=matlab,contour/draw color={mapped color!50!black},dashed,line width=0.5pt] table{poly_sirt.dat}; 
\end{axis}
\end{tikzpicture}

\begin{tikzpicture}
\begin{axis}[%
width=0.36\linewidth,
height=0.36\linewidth,
title style={at={(0.5,0.85)}},
x label style={at={(0.5,-0.05)}},
y label style={at={(-0.1,0.5)},rotate=0},
xmin=0,xmax=0.5,ymin=0.2,ymax=1.4,
xtick={0.1,0.4},
ytick={0.3,1.3},
xlabel={$\theta_1$},ylabel={$\nu_1$},title={$\mathcal{Q}_1{}^\sharp f_{\bX_2}$},
]
\pgfplotsset{
    contour/every contour plot/.style={labels=false},
}
\addplot[contour prepared,contour prepared format=matlab,contour/draw color={black},opacity=0.3,line width=2pt] table{pullback_2.dat};
\addplot[contour prepared,contour prepared format=matlab,line width=0.5pt] table{pullback_approx_2.dat}; 
\end{axis}
\end{tikzpicture}
\hspace{-1em}
\begin{tikzpicture}
\begin{axis}[%
width=0.36\linewidth,
height=0.36\linewidth,
title style={at={(0.5,0.85)}},
x label style={at={(0.5,-0.05)}},
y label style={at={(-0.1,0.5)},rotate=0},
xmin=0,xmax=1,ymin=0.4,ymax=1.5,
xtick={0.2,0.8},
ytick={0.5,1.4},
xlabel={$\theta_1$},ylabel={$\nu_1$},title={$(\mathcal{Q}_1\circ\mathcal{Q}_2)^\sharp f_{\bX_3}$},
legend style={at={(0.85,1)},anchor=north west,fill=none},
legend cell align={left},
legend image post style={
    sharp plot,
    draw=\pgfkeysvalueof{/pgfplots/contour/draw color},
},
legend entries={exact,approximation}
]
\pgfplotsset{
    contour/every contour plot/.style={labels=false},
}
\addplot[contour prepared,contour prepared format=matlab,contour/draw color={black},opacity=0.3,line width=2pt] table{pullback_3.dat};
\addplot[contour prepared,contour prepared format=matlab,line width=0.5pt] table{pullback_approx_3.dat}; 
\end{axis}
\end{tikzpicture}
\hspace{-1em}\vspace{-1em}

\caption{CSIR example with $K=1$. Top row: the thick grey contours show the bridging densities, the solid contours  show the self-reinforced approximations, the dash contours on the right show the single-layered approximation. 
Bottom row: pullback densities used in the self-reinforced construction (thick grey) and the corresponding sparse polynomial approximations (solid).}\label{fig:csir_single}
\end{figure}
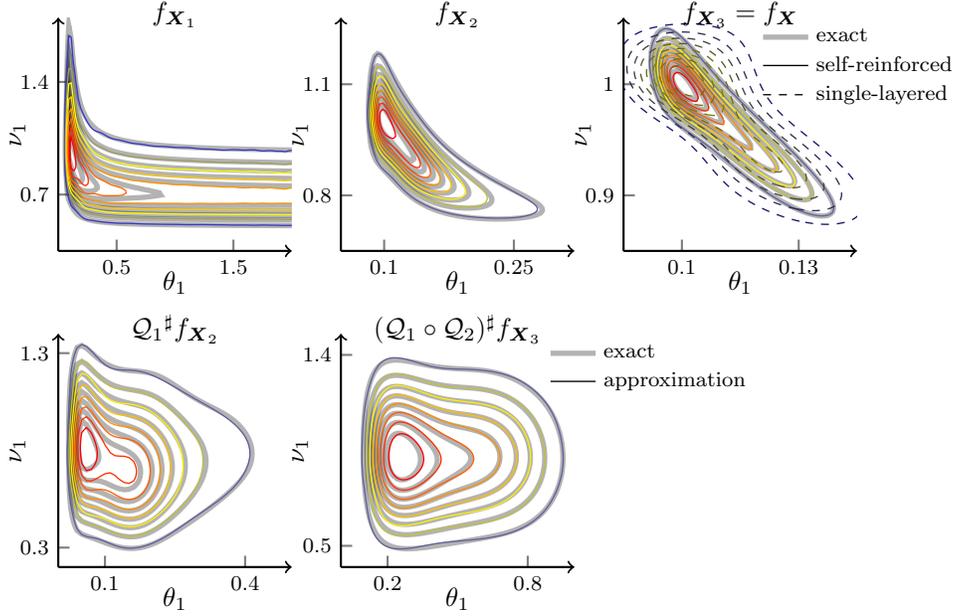

\paragraph{Sparse polynomials versus tensor trains} We test the dimension scalability of the self-reinforced sparse-polynomial construction of the KR rearrangement and compare its performance with the TT-based construction of \cite{cui2021deep}. We test CSIR models with $K \in \{1,2,3,4\}$ numbers of compartments. For both  sparse-polynomial-based and TT-based constructions, we use a uniform reference measure and Legendre polynomials with a maximum order $30$ on each coordinate. To build self-reinforced KR rearrangements by adaptive tempering, Alg.~\ref{alg:dirt_adapt} uses an initial tempering parameter $\beta_1=10^{-3}$ and the adaptation threshold $\eta = 0.1$.
We first set up the TT-based construction, in which the polynomial-based KR rearrangements in Alg.~\ref{alg:dirt_adapt} are replaced by TT-based KR rearrangements. At each layer, the TT approximation is carried out with $2$ iterations of TT-Cross, SVD truncation tolerance $10^{-2}$ and maximal TT rank $10$. In the sparse-polynomial-based construction, we match the number of density evaluations per layer to that of the TT-based one rather than setting a stopping tolerance on the estimated error. The Hellinger distances of the resulting maps, the number of density evaluations used and the number of layers in the composite maps for each parameter dimension ($d = 2K$) are reported in Fig.~\ref{fig:csir_poly_tt}. Compared to the TT-based construction, the polynomial-based construction is slightly less accurate for lower dimensional settings and the accuracy gap is reduced with increasing dimension.

\begin{figure}[h!]
\centering
\begin{tikzpicture}
\begin{axis}[%
width=0.48\linewidth,
height=0.35\linewidth,
xmode=normal,
xmin=1,xmax=9, 
ymode=log,
xlabel={$d$},
ylabel={$D_H(f_X,f_{\hat X})$},
legend style={at={(0.98,0.02)},anchor=south east},
nodes near coords,point meta = explicit symbolic,
x label style={at={(0.5,-0.1)}},
]
\addplot+[error bars/.cd,y dir=both,y explicit] coordinates{
(2, 0.0014) +- (2, 0.000466) 
(4, 0.0031) +- (4, 0.000742) 
(6, 0.0106) +- (6, 0.0012)  
(8, 0.0208) +- (8, 0.0042) 
}; \addlegendentry{TT};
\addplot+[error bars/.cd,y dir=both,y explicit] coordinates{
(2, 0.0053) +- (2, 0.0020) 
(4, 0.0063) +- (4, 0.0013)
(6, 0.0153) +- (6, 0.0026)
(8, 0.0286) +- (8, 0.0053)
}; \addlegendentry{poly};
\end{axis}
\end{tikzpicture}
\begin{tikzpicture}
\begin{axis}[%
width=0.48\linewidth,
height=0.35\linewidth,
xmode=normal,
xmin=1,xmax=9, 
ymode=log,
ymin=5e2,
xlabel={$d$},
ylabel={$N_{\rm evals}$},
legend style={at={(0.98,0.02)},anchor=south east},
x label style={at={(0.5,-0.1)}},
]
\addplot+[nodes near coords,point meta = explicit symbolic, every node near coord/.style={anchor=north west,inner sep=2pt}] coordinates{
(2, 820) +- (2, 16) [3]
(4, 29140) +- (4, 573)  [5]
(6, 107787) +- (6, 7390)  [6]
(8, 256800) +- (8, 26280)  [7]
}; \addlegendentry{TT};
\addplot+[nodes near coords,point meta = explicit symbolic, every node near coord/.style={anchor=south east,inner sep=2pt}] coordinates{
(2, 1330)  +- (2, 36) [4]
(4, 29097)  +- (4, 608) [5]
(6, 109800)  +- (6, 4223) [6]
(8, 259148) [7]
}; \addlegendentry{poly};
\end{axis}
\end{tikzpicture}\vspace{-12pt}
\caption{CSIR example. The Hellinger distance (left) and the number of density evaluations (right) for each parameter dimension ($d = 2K$). The labels on the right plot indicate the number of layers in the composite map.}\label{fig:csir_poly_tt}
\end{figure}
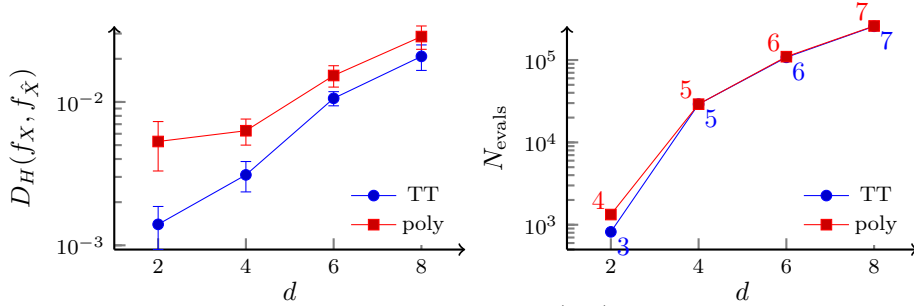

\subsection{Elliptic PDE}
We then consider the parameter estimation problem of an elliptic PDE commonly arisen in groundwater modelling. Given a problem domain $D=[0,1]^2$, we want to estimate unknown diffusivity field $\kappa(s,\bX), s\in D$ from partial observation of the water table $u(s,\bX)$, which is a function that satisfies the PDE
\begin{align}
    -\nabla \cdot ( \kappa(s,\bX) \nabla u(s,\bX)) & = 0, \quad s\in(0,1)^2 ,  \label{eq:pde}
\end{align}
with Dirichlet boundary conditions
$u|_{s_1=0} = 1+s_2/2$ and $u|_{s_1=1} = -\sin(2\pi s_2)-1$
imposed horizontally and no-flux boundary conditions $\partial u / \partial s_2|_{s_2=0} = \partial u / \partial s_2|_{s_2=1} = 0$ imposed vertically.
For each realization of $\bX$, we apply the Galerkin method with continuous, bilinear finite elements to numerically solve \eqref{eq:pde}. The finite element solution $u_h$ is computed on a uniform rectangular grid on $D$ with a mesh size $h = 1/64$ along each of the
coordinates of $D$.

We model the logarithm of the diffusivity field using a zero mean uniform prior random field with the covariance operator defined by the Mat\'ern covariance function
\[
  C(s,t) = \frac{2^{1-\nu}}{\Gamma(\nu)}\Big(\sqrt{2\nu} \frac{\|s-t\|_2}{\ell}\Big)^{\nu} K_{\nu}\Big(\sqrt{2\nu}\frac{\|s-t\|_2}{\ell}\Big), \quad s,t \in D,
\]
where $\nu=2$ and $\ell=1$.%
Using the Karhunen--L\'oeve (KL) expansion, $\log\kappa(s,\bX)$ yields a finite-dimensional approximation
\[
    \log\kappa(s,\bX) \approx \sum_{k = 1}^{d} X_k \sqrt{\omega_k} \varphi_k(s) ,
\]
where $\{\varphi_k(s), \omega_k\}$ is the $k$th eigenpair of the covariance operator in the descending order of eigenvalues and each random coefficient $X_k$ follows the uniform prior with zero mean and unit variance, i.e., $X_k \sim \mathrm{uniform}(-\sqrt{3},\sqrt{3})$.

To setup the observation model, we measure the water table $u(s,X)$ at
$m = 15 \times 15$ locations defined as the vertices of a uniform Cartesian grid
on $D=[0,1]^2$ with grid size $1/(\sqrt m+1)$. Measurements are
corrupted by i.i.d. Gaussian noise with zero mean and variance $\sigma_n^2 = 10^{-2}$.
This leads to the parameter-to-observable map
\begin{equation}\label{eq:pde-data}
y_i = G_i(\bx) + \eta_i, \quad G_{i_1 + (i_2-1)\sqrt m}(\bx) = u_h((\frac{i_1}{\sqrt m+1}, \frac{i_2}{\sqrt m+1}),\bx) 
\end{equation}
for $i_1,i_2=1,\ldots,15$, where $\eta_i \sim \mathcal{N}(0,\sigma_n^2)$.
Synthetic data is generated as $\by = G(\bx_{\rm true}) + \eta$ with $\bx_{\rm true} = (0.5,\ldots,0.5)$.

\paragraph{Sparse polynomials versus tensor trains} We use varying number of KL basis functions, $d \in \{6, 8, 10, 12, 14, 16\}$, to test the dimension scalability of the self-reinforced sparse-polynomial construction of the KR rearrangement and compare its performance with the tensor-train (TT)  construction of \cite{cui2021deep}. For both  sparse-polynomial-based and TT-based constructions, we use a uniform reference measure and Legendre polynomials with a maximum order $20$ on each coordinate. To build self-reinforced KR rearrangements by adaptive tempering, Alg.~\ref{alg:dirt_adapt} uses an initial tempering parameter $\beta_1=10^{-3}$ and the adaptation threshold $\eta = 0.5$.
At each layer, the TT approximation is carried out with $2$ iterations of TT-Cross, SVD truncation tolerance $10^{-2}$ and maximal TT rank $5$. In the sparse-polynomial-based construction, we match the number of density evaluations per layer to that of the TT-based one rather than setting a stopping tolerance on the estimated error. The Hellinger divergence of the resulting maps, the number of density evaluations used and the number of layers in the composite maps for each KL dimension $d$ are reported in Fig.~\ref{fig:elliptic_poly_tt}. Compared to the TT-based construction, the polynomial-based construction is more accurate for all $d$, and requires less number of density evaluations for $d\leq14$. For $d = 16$, the TT-based construction needs less number of density evaluations and the accuracy gap with the polynomial-based construction becomes smaller.

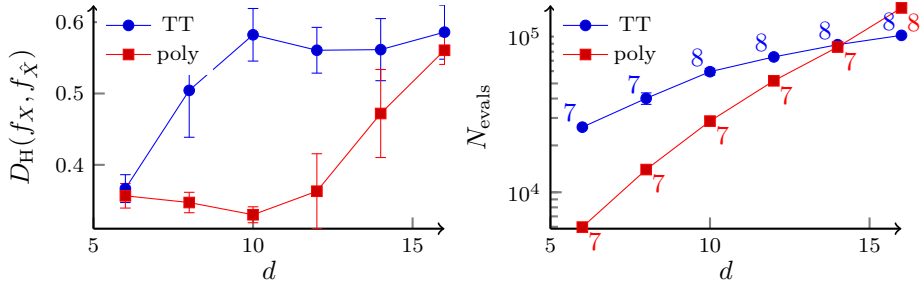
\begin{figure}[h]
\centering
\begin{tikzpicture}
\begin{axis}[%
width=0.48\linewidth,
height=0.35\linewidth,
xmode=normal,
xmin=5,xmax=16,
ymode=normal,
xlabel={$d$},
ylabel={$D_{\rm H}(f_X,f_{\hat X})$},
legend style={at={(0.01,1)},anchor=north west},
nodes near coords,point meta = explicit symbolic,
every node near coord/.style={anchor=south east},
x label style={at={(0.5,-0.1)}},
]
\addplot+[error bars/.cd,y dir=both,y explicit] coordinates{
(6 , 0.366736)  +- (6 , 0.0195021) 
(8 , 0.504362)  +- (8, 0.0656966) 
(10 , 0.582281)  +- (10, 0.0368477) 
(12,  0.560613)  +- (12 , 0.03207)
(14,  0.561472)  +- (14,  0.0435164)
(16, 0.585975)  +- (16 , 0.0378927) 
}; \addlegendentry{TT};
\addplot+[error bars/.cd,y dir=both,y explicit] coordinates{
(6  , 0.356621) +- (6,  0.0171799) 
(8 , 0.347226) +- (8,  0.014316) 
(10 , 0.330084) +- (10, 0.0111817)
(12,  0.363058) +- (12, 0.0525189) 
(14,  0.471964) +- (14,    0.0616972)
(16, 0.560603) +- (16, 0.0198323)
}; \addlegendentry{poly};
\end{axis}
\end{tikzpicture}
\begin{tikzpicture}
\begin{axis}[%
width=0.48\linewidth,
height=0.35\linewidth,
xmode=normal,
xmin=5,xmax=16,
ymode=log,
xlabel={$d$},
ylabel={$N_{\rm evals}$},
legend style={at={(0.01,1)},anchor=north west},
x label style={at={(0.5,-0.1)}},
]
\addplot+[nodes near coords,point meta = explicit symbolic, every node near coord/.style={anchor=south east,inner sep=2pt},error bars/.cd,y dir=both,y explicit] coordinates{
(6, 26198)  +- (6, 374) [7]
(8, 40103)  +- (8, 3402) [7]
(10, 59507)  +- (10, 2986) [8]
(12, 74050)  +- (12, 643) [8]
(14, 88722)  +- (14, 506) [8]
(16, 101885)  +- (16, 609) [8]
}; \addlegendentry{TT};
\addplot+[nodes near coords,point meta = explicit symbolic, every node near coord/.style={anchor=north west,inner sep=2pt},error bars/.cd,y dir=both,y explicit] coordinates{
(6, 5972) +- (6, 138) [7]
(8, 13990) +- (8, 282) [7]
(10, 28643) +- (10, 1286) [7] 
(12, 52019) +- (12, 2209) [7]
(14, 85680) +- (14, 0) [7]
(16, 152887) +- (16, 6460) [8] 
}; \addlegendentry{poly};
\end{axis}
\end{tikzpicture}\vspace{-12pt}
\caption{Elliptic example. The Hellinger distance (left) and the number of density evaluations (right) for each parameter dimension. The labels on the right plot indicate the number of layers in the composite map.}\label{fig:elliptic_poly_tt}
\end{figure}

\paragraph{Tempering versus data-batching} We also compare the tempering approach (cf. Example \ref{ex:bridge1}) and the data-batching approach (cf. Example \ref{ex:bridge2}) for building the composite maps based on sparse polynomials. The Hellinger divergence of the resulting maps, the number of density evaluations used and the number of layers in the composite maps for $d \in \{6,12\}$ are reported in Table~\ref{tab:elliptic_bridging}. Despite that the data-batching approach is marginally more accurate, it doubles the number of layers in the composite map and also doubles the number of density evaluations needed.

\begin{table}[h]
\footnotesize
\caption{Elliptic example. The Hellinger distance,  number of density evaluations and number of layers in the composite map for different parameter dimensions and different bridging densities.}\vspace{-6pt}\label{tab:elliptic_bridging}
\centering
\begin{tabular}{c|ccc|ccc}
$d$ & \multicolumn{3}{c|}{tempering}                        & \multicolumn{3}{c}{data-batching} \\
    & $D_H(f_X,f_{\hat X})$ & $N_{evals}$ & $\#\{\beta\}$ & $D_H(f_X,f_{\hat X})$ & $N_{evals}$ & $\#\{\beta\}$ \\\hline
6   & $0.36\pm 0.017$ & $5972 \pm 139$ & 7 &  $0.30\pm0.023$  & $14378\pm568$ & $16$  \\
12  & $0.36\pm 0.053$ & $52019\pm 2209$ & 7 & $0.33\pm 0.017$ & $138756\pm 4557$ & $16$ \\
\end{tabular}
\end{table}

\paragraph{Comparison of single-layered and self-reinforced constructions} To make the single-layered construction feasible, we only use $9$ observations at $m = 3\times 3$ locations specified in \eqref{eq:pde-data}. To demonstrate the comparison for changing posterior concentration and parameter dimensionality, we consider three experiments: a reference experiment with $d = 6$ and a large observation noise variance $\sigma_{\rm n}^2 = 10^{-1}$; an experiment with an increasing dimension $d = 12$ and the same variance $\sigma_{\rm n}^2 = 10^{-1}$; and an experiment with a dimension $d = 6$ and a reduced variance $\sigma_{\rm n}^2 = 10^{-2}$. For both constructions, we set a stopping tolerance for the least square procedure to be $0.01$ and a maximum cardinality of $3.2 \times 10^4$ for polynomial basis functions. The Hellinger distances, the numbers of density evaluations and the total numbers of basis functions used by both constructions are reported in Table \ref{tab:elliptic_single}. In this case, the self-reinforced construction performs marginally more efficient than the single-layered construction on the reference experiment. With either increasing dimensionality or decreasing observation noise variance (so the posterior is more concentrated), the self-reinforced construction achieves much better computational efficiency than the single-layered construction. For both $d = 12$ and $\sigma_{\rm n}=10^{-2}$ experiments, the single-layered construction uses up the maximum cardinality of basis functions. For $d = 12$, it achieves comparable accuracy to the self-reinforced construction with $6.5$ times as many basis functions. For $\sigma_{\rm n}=10^{-2}$, its error is about an order of magnitude higher than that of the self-reinforced construction.

\begin{table}[h]
\footnotesize
\caption{Elliptic example. The Hellinger distances, numbers of density evaluations and total numbers of basis functions for single-layered and self-reinforced constructions. }\vspace{-6pt}\label{tab:elliptic_single}
\centering
\setlength\tabcolsep{2.5pt}
\begin{tabular}{r|cccc|ccc}
$(d,\sigma_{\rm n}^2)$ & \multicolumn{4}{c|}{self-reinforced}                        & \multicolumn{3}{c}{single-layered} \\
    & $D_H(f_X,f_{\hat X})$ & $N_{evals}$ & $\#$ basis & \hspace{-6pt}$\#\{\beta\}$ & $D_H(f_X,f_{\hat X})$ & $N_{evals}$ & $\#$ basis \\\hline
$(6, 10^{-1})$   & 4.1e-3$\pm$8.6e-4 & 2.5e3$\pm$2.8e2 & 8.6e2$\pm$8.4e1 & 2 & 4.9e-3$\pm$1.4e-3  & 9.5e3$\pm$8.7e2 & 3.2e3$\pm$2.3e2 \\
$(12, 10^{-1})$   & 4.8e-3$\pm$1.1e-4 & 1.4e4$\pm$2.4e3 & 4.9e3$\pm$8.0e2 & 2 & 6.6e-3$\pm$4.6e-4  & 9.6e4 & 3.2e4 \\
$(6, 10^{-2})$  & 4.3e-3$\pm$8.5e-4 & 2.5e4$\pm$3.3e3 & 8.1e3$\pm$7.9e2 & 3 & 7.6e-2$\pm$9.4e-3 & 1.28e5 & 3.2e4 
\end{tabular}
\end{table}

\subsection{Linear elasticity analysis of a wrench}

Lastly, we consider a more complicated linear elasticity problem \cite{lam2020multifidelity,smetana2020randomized}, in which the displacement field $u:\mathcal{D} \rightarrow \mathbb{R}^{2}$ of a physical body $\mathcal{D} \subset \R^2$ is in a equilibrium state subject to external forces $f$. Given the strain field $\varepsilon(u) = \frac{1}{2}(\nabla u + \nabla u^\lowsup{\top})$, the displacement $u$  satisfies the PDE 
\[
\mathrm{div}\left(K:\varepsilon(u)\right) = f \quad \text{on} \quad \mathcal{D} \subset \mathbb{R}^{2},
\]
where $K$ is the Hooke tensor such that
\[
K:\varepsilon(u) = \frac{\exp(\theta)}{1+\nu} \varepsilon(u) + \frac{\nu \exp(\theta)}{1-\nu^2} \mathrm{trace}(\varepsilon(u))\begin{pmatrix} 1&0 \\0&1\end{pmatrix}
,\]
$\nu=0.3$ is Poisson's ratio, and $\exp(\theta):\mathcal{D}\rightarrow\R_{>0}$ is spatially varying Young's modulus. 
We aim to estimate the logarithm of Young's modulus, $\theta:\mathcal{D}\rightarrow\R$, using partial observations of the displacement. 

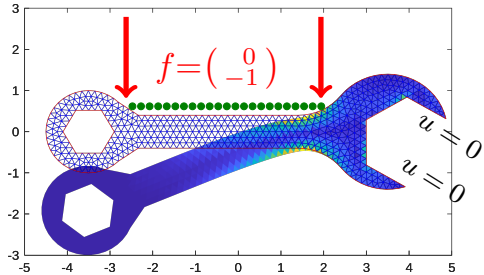
\begin{figure}[h]
\centering
\begin{tikzpicture}[scale=1.4, every node/.style={scale=1.3}]
\node at (0,0.25) {\includegraphics[width=0.35\linewidth,height=0.21\linewidth]{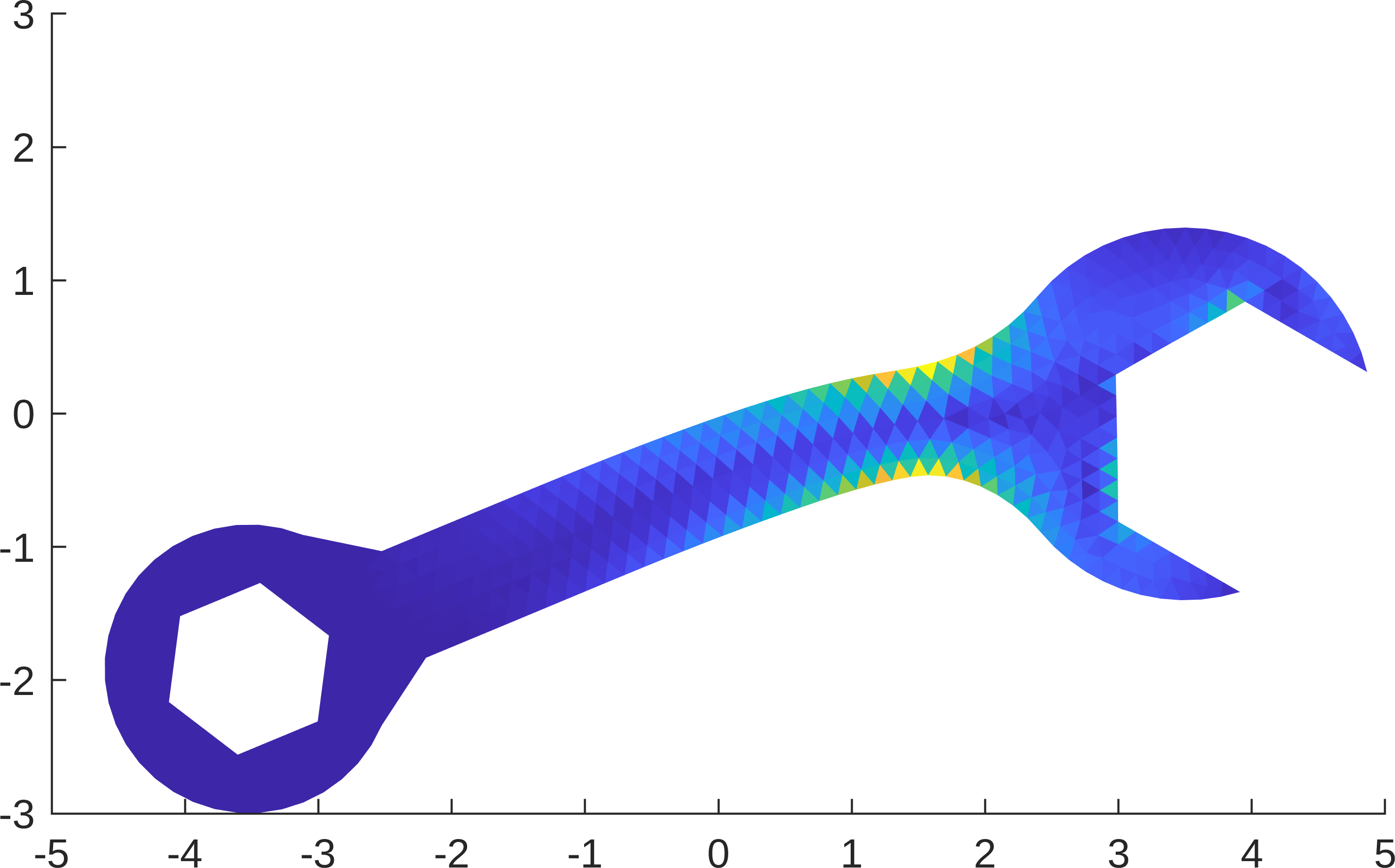}};
\node at (0,0.25) {\includegraphics[width=0.35\linewidth,height=0.21\linewidth]{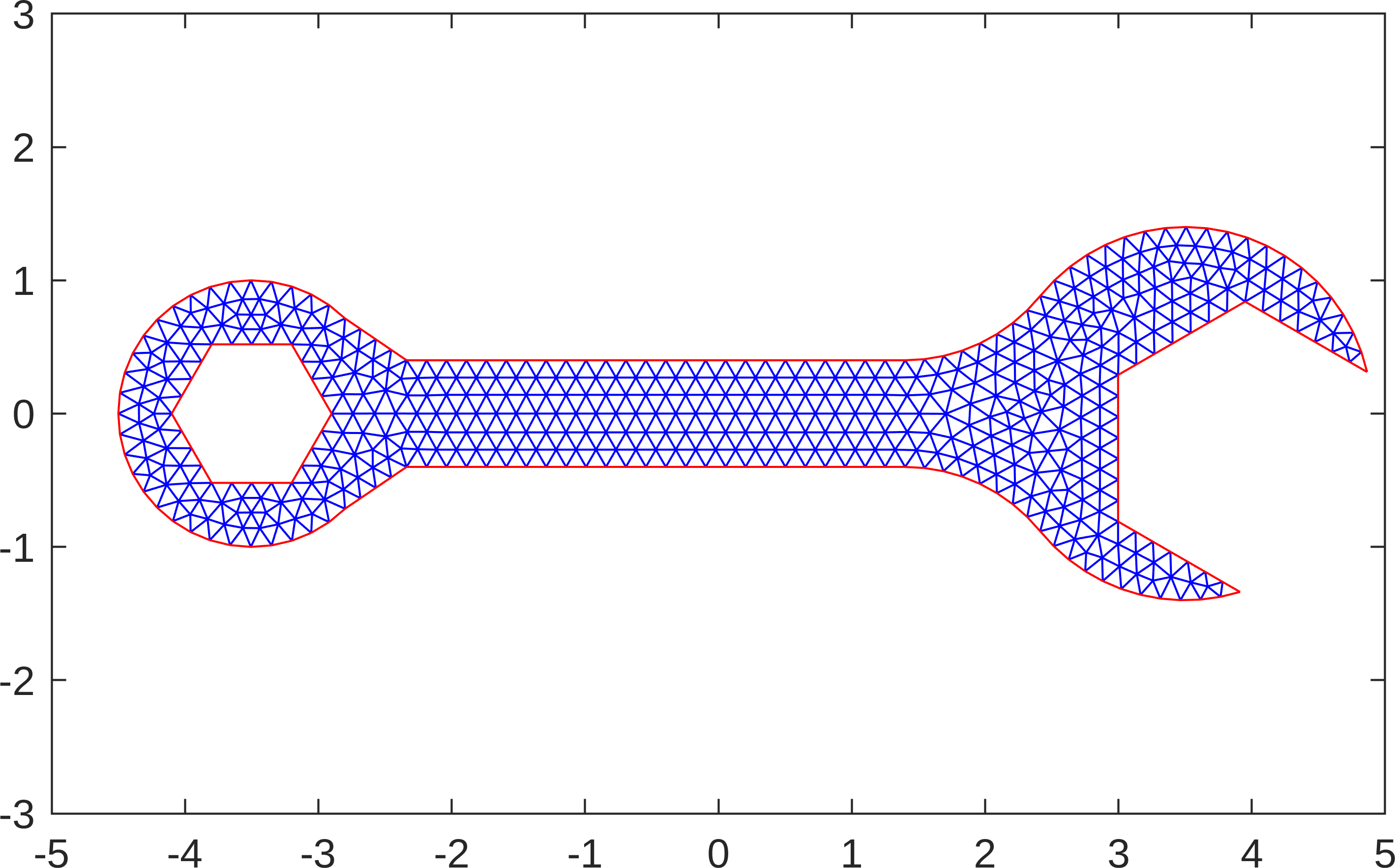}};  
\node[rotate=-32] at (2.4*0.8,-0.2*0.7) {\footnotesize$u=0$};
\node[rotate=-32] at (2.6*0.8,0.4*0.7) {\footnotesize$u=0$};
\draw[->,red,line width=2pt] (-1,1.4) -- (-1,0.63);
\draw[->,red,line width=2pt] (0.85,1.4) -- (0.85,0.63);
\node[red,anchor=west] at (-1.2*0.7,1.3*0.7) {$f {=} \begin{psmallmatrix}\;\;0\\{-}1\end{psmallmatrix}$};
\foreach \x in { -1.88965, -1.74012, -1.59059, -1.44106, -1.29153, -1.142, -0.992471, -0.842941, -0.693412, -0.543882, -0.394353, -0.244824, -0.0952941, 0.0542353, 0.203765, 0.353294, 0.502824, 0.652353, 0.801882, 0.951412, 1.10094, 1.25047, 1.4} {
\node[ellipse,draw=green!50!black,fill=green!50!black,inner sep=0.7pt,anchor=center] at (0.545*\x+0.09,0.55) {};
}
\node at (0,-2) {};
\end{tikzpicture}\vspace{-4em}
\caption{Setup of the linear elasticity example.}
\label{fig:wrench}
\end{figure}

We let the unknown log-Young's modulus follow a Gaussian random field prior $\theta \sim \mathcal{N}(0,C)$, where $C(s,t) = \exp(-\|s-t\|_2^2)$ is a Gaussian covariance function on $\mathcal{D}\times \mathcal{D}$.
After finite element discretization, $\theta$ is replaced with a piecewise constant field whose elementwise values are gathered in a vector $\theta\in\R^{d_\theta}$ where $d_\theta=925$ is the number of elements in the mesh. The parameter $\theta$ thus follows a centered Gaussian prior with covariance matrix $\Sigma_{ij}=C(s_i,s_j)$, where $s_i\in\mathcal{D}$ is the center of the $i$-th element.
We compute the numerical solution $u^\lowsup{h}=u^\lowsup{h}(\theta)$ by the Galerkin projection of $u=u(\theta)$ onto the space of continuous piecewise affine functions over a triangular mesh.
The domain $\mathcal{D}$, the mesh, the boundary conditions, and a sample von Mises stress of the solution are shown in Fig.~\ref{fig:wrench}.
We observe the vertical displacements $u^h_2$ of the $m=26$ points of interests located along the green line where the force $f$ is applied, see Fig.~\ref{fig:wrench}. The perturbed observations are $y=u^h_2+e$ where $e$ is a zero-mean $\mathbb{H}^{1}$-normal noise with the signal-to-noise-ratio $10$.

Instead of directly operating with the high-dimensional discretized parameter $\theta$, we apply the data-free dimension reduction method of \cite{cui2021conditional,cui2021data} to compress the parameter dimensionality. We first construct a sensitivity matrix $H \in \R^{d_\theta \times d_\theta}$ by integrating the Fisher information matrix over the prior distribution, and then compute the eigenpairs $(\varphi_i,\nu_i)$---in descending order of eigenvalues $\nu_i$---of the matrix pencil $(H, C^{-1})$. This gives a reparametrization \(\smash{\theta = \sum_{i = 1}^{d_\theta} \varphi_i x_i}\), where each $x_i$ follows the zero mean standard Gaussian prior. By truncating the reparametrization, i.e., $\smash{\theta = \sum_{i = 1}^{d} \varphi_i x_i}$ with $d < d_\theta$, according to the eigenvalues $\nu_i$, we obtain reduced-dimensional parameterization of the random field. 

We then construct the self-reinforced KR rearrangements to approximate the reduced-dimensional posterior distribution. We truncated the Gaussian prior at $\pm 5$ standard deviations, which leads to a truncated domain of approximation $[-5,5]^{d}$. For both  sparse-polynomial-based and TT-based constructions, we follow the exactly same setup as in the elliptic PDE example.
The Hellinger distances of the resulting maps, the numbers of density evaluations used and the number of layers in the composite maps for each reduced parameter dimension $d$ are reported in Table~\ref{tab:wrench}. Compared to the TT-based construction, the polynomial-based construction has similar accuracy and uses a similar number of density evaluations for $d = 7$. For $d=11$, the polynomial-based construction is slightly less accurate than the TT-based construction. 

\begin{table}[h]
\footnotesize
\caption{Linear elasticity example. The Hellinger distances, numbers of density evaluations and numbers of layers in the composite maps.} \vspace{-6pt}\label{tab:wrench}
\centering
\begin{tabular}{c|ccc|ccc}
$d$ & \multicolumn{3}{c|}{TT}                        & \multicolumn{3}{c}{sparse polynomials} \\
  & $D_H(f_X,f_{\hat X})$ & $N_{evals}$ & $\#\{\beta\}$ & $D_H(f_X,f_{\hat X})$ & $N_{evals}$ & $\#\{\beta\}$ \\\hline
7 & 0.21$\pm$0.024 & 44667$\pm$2083   & 7       & 0.22$\pm$0.020 & 45863$\pm$1077 & 7 \\
11& 0.41$\pm$0.083  & 92155$\pm$385    & 8       & 0.49$\pm$0.033  & 108162$\pm$6473& 8 \\
\end{tabular}
\end{table}

\subsection{Concluding remarks}
In numerical examples, the self-reinforced KR rearrangements constructed using sparse polynomials have similar performance compared to those based on TT. The use of self-reinforced construction significantly improves the approximation power of sparse polynomial approximation. In the CSIR model, without using the self-reinforced construction, it is even infeasible to go beyond two compartments. 
For the elliptic PDE, the single-layered construction can only be applied to cases with less concentrated posterior densities---by increasing the observation noise and reducing the number of observations---while the self-reinforced construction does not suffer from these difficulties. 
We also demonstrate that data-batching is a useful approach for building bridging densities. Although the data-batching approach uses more layers than the tempering approach on the elliptic PDE example, it can be potentially more advantageous for problems where the cost of evaluating the likelihood is proportional to the data dimension. In addition, the data-batching approach is naturally more suitable for sequential experimental designs, where data are observed sequentially; see \cite{huan2015numerical} and reference therein for details.

\section*{Acknowledgments}
TC acknowledges support from the Australian Research Council under the grant DP210103092. SD acknowledges support from the Engineering and Physical Sciences Research Council New Investigator Award EP/T031255/1.
OZ acknowledges support from the ANR JCJC project MODENA (ANR-21-CE46-0006-01).

\appendix

\section{Polynomial bases and distribution functions}\label{sec:polys}
To provide an algebraically exact implementation of the distribution function, here we provide a recipe for the indefinite integral of the form 
\(
    \int \psi_{j}(y)  \lambda(y) \dd y 
\)
for a range of polynomials and weight functions. In the following, we let the weight function be normalized, i.e., $\int_\mathcal{I} \lambda(x) \dd x = 1$, and also let the basis functions be normalized, i.e.,  $\langle \psi_j, \psi_k \rangle_\lambda = \delta(j,k)$.

\subsection{Chebyshev polynomials on \texorpdfstring{$\mathcal{I} = (-1,1)$}{}}\label{sec:cheby1st}

The normalized weight function and the normalized  Chebyshev polynomials of the first kind are defined by 
\begin{equation}\label{eq:cheby1st}
    \lambda(x) = \frac{1}{\pi \sqrt{1-x^2}}, \quad \psi_j(x) = w_j \cos(j \cos^{-1}(x)), \quad j = 0, 1, \ldots,
\end{equation}
with $w_0 = 1$ and $w_j = \sqrt{2}$ for $j \geq 1$. We have the following indefinite integral 
\[
\int  \psi_{j}(x)  \lambda(x) \dd x = \begin{cases} - \frac1\pi \cos^{-1}(x), & j = 0 \vspace{4pt}\\ -\frac{\sqrt{2}}{j \pi} \sin(j \cos^{-1}(x)), & j \geq 1 \end{cases}.
\]

The normalized weight function and the normalized Chebyshev polynomials of the second kind are defined by 
\begin{equation}\label{eq:cheby2nd}
    \lambda(x) = \frac{2\sqrt{1-x^2}}{\pi}, \quad \psi_j(x) = \frac{\sin((j+1) \cos^{-1}(x))}{\sin(\cos^{-1}(x))}, \quad j = 0, 1, \ldots.
\end{equation}
We have the following indefinite integral 
\[
\int  \psi_{j}(x)  \lambda(x) \dd x = \begin{cases} \frac1\pi \left(\frac12 \sin(2 \cos^{-1}(x)) - \cos^{-1}(x) \right), & j = 0 \vspace{4pt}\\ \frac1\pi \left(\frac{1}{j+2}\sin((j+2) \cos^{-1}(x)) - \frac{1}{j}\sin(j \cos^{-1}(x))\right), & j \geq 1 \end{cases}.
\]

\subsection{Legendre polynomials on \texorpdfstring{$\mathcal{I} = (-1,1)$}{}}\label{sec:legendre}

For a density function represented by Legendre polynomials, it is easier to use the collocation method and the Chebyshev polynomials of the second kind to compute its distribution function. 
Using the conditional density in \eqref{eq:one_pdf} as the example, we have
\[
    f_{\hat X_{t} | \hat\bX_{[t{-}1]}}(x_t | \bx_{[t{-}1]}) = \frac1{2\zeta} q(x_t), \quad q(x_t) = c + \sum_{\ell = 1}^{r} \Big( \sum_{\bk \in \mathcal{K}} \psi_{k_t}(x_t)\mC_{\bk\ell}\Big)^2  ,      
\]
where $(\psi_{k_t})$ are Legendre polynomials with the maximum order $n_t$. We define collocation points as the roots of $\varphi_{2n_t+1}(x)$, where $\varphi_j$ is the Chebyshev polynomials of the second kind \eqref{eq:cheby2nd}. This gives $x^j = \cos\big(\frac{j \pi} {2n_t+2}\big)$, $j = 1, \ldots, 2n_t +1$.
This way, by evaluating the function $q(x_t)$ on the collocation points, one can apply the collocation method to represent $\frac1{2\zeta}q(x_t)$ using the Chebyshev basis, $\frac1{2\zeta}q(x_t) = \sum_{j=0}^{2n_t} a_j \, \varphi_j(x_t)$,
where the coefficients $(a_j)_{j=1}^{2n_t}$ can be computed by the fast Fourier transform with $O(n_t\log(n_t))$ operations. Then, the distribution function of $\hat X_t | \hat \bX_{[t-1]}$ becomes
\begin{align*}
    \mathcal{F}_{\hat X_{t} | \hat\bX_{[t{-}1]}} (x_t | \bx_{[t{-}1]}) &  = \sum_{j=0}^{2n_t} \frac{a_j}{j+1} \big(\cos( (j+1) \cos^{-1}(x_t)) + (-1)^{j} \big). 
\end{align*}

\subsection{Hermite polynomials on \texorpdfstring{$\mathcal{I} = (-\infty, \infty)$}{}}\label{sec:hermite}

For Hermite polynomials, the normalized weight function is $\lambda(x) = \frac1{\sqrt{2\pi}} \exp(- \frac12 x^2)$ and the normalized polynomials are $\psi_j(x) = (j !)^{-1/2} p_j(x), j = 0, 1, \ldots$, where $p_j(x)$ are the probabilist's Hermite polynomials specified by the three-term recurrence 
\[
    p_{j+1}(x) = x p_j(x) - j p_{j-1}(x), \quad p_1(x) = x, \quad p_0(x) = 1.
\]
Since the recurrence can also be specified as
\(
    p_{j+1}(x) = x p_j(x) - p_{j}'(x),
\)
we have $( p_{j}(x)\lambda(x) )' = -p_{j+1}(x) \lambda(x)$, which leads to the following indefinite integral 
\[
\int  \psi_{j}(x)  \lambda(x) \dd x = \begin{cases} \frac12 \mathrm{erf}({x}/{\sqrt{2}}), & j = 0 \vspace{4pt}\\ - \frac{w_j}{w_{j-1}}\psi_{j}(x)\lambda(x) , & j \geq 1 \end{cases}.
\]

\subsection{Laguerre polynomials on \texorpdfstring{$\mathcal{I} = [0, \infty)$}{}}\label{sec:laguerre}

For Laguerre polynomials, the normalized weight function is the density of the exponential distribution,
\(
    \lambda(x) = \exp(-x) ,
\)
and the normalized polynomials are specified by the three-term recurrence 
\[
    (j+1)\psi_{j+1}(x) = (2j + 1 - x) \psi_j(x) - j \psi_{j-1}(x), \quad \psi_1(x) = 1-x, \quad \psi_0(x) = 1.
\]
Since the recurrence can also be specified as $x\psi'_j(x) = j ( \psi_j(x) - \psi_{j-1}(x) )$ and the identity $\psi'_j(x) = - \sum_{k = 0}^{j-1} \psi_k(x)$ for $j\geq 1$, we have the following indefinite integral 
\[
\int  \psi_{j}(x)  \lambda(x) \dd x = \begin{cases} -\exp(-x), & j = 0 \vspace{4pt}\\ \frac1j x \exp(-x) \sum_{k = 0}^{j-1} \psi_k(x) , & j \geq 1 \end{cases}.
\]

\section{Implementing the domain mapping in Section \ref{sec:rd}}\label{sec:domain_mapping}
Suppose we have the target density $f_X(x)$, $x \in \R$, and an invertible transformation $z(x)$ (cf. Table \ref{table:mappings}). We  want to show that the following two approaches of approximating the target density using domain mapping are equivalent: ({\romannumeral 1}) applying the square-root approximation procedure of Section \ref{sec:sqrt_approx} to $f$ using the induced weight  $\lambda(x)$ defined in \eqref{eq:extended_weight} and the extended basis functions $(\phi_j)_{j=0}^{\infty}$ defined in \eqref{eq:extended_basis}; and ({\romannumeral 2}) applying the square-root approximation procedure to the pullback density $f_Z(z):=f_X(x(z))x'(z)$ with weight $\lambda^{\rm ref}(z)$ and polynomials $(\psi_j)_{j=0}^\infty$ defined on $(-1,1)$. 

We first the direct approximation approach ({\romannumeral 1}). Using the square-root approximation procedure, one first expresses the target density as 
\begin{equation}\label{eq:inf_ratio}
f_X(x) = p(x) \lambda(x), \quad p(x) = \displaystyle\frac{f_X(x)}{\lambda(x)},
\end{equation}
and then projects the function $\sqrt{p(x)}$ onto the extended basis functions $(\phi_j)_{j=0}^{\infty}$ with respect to the induced weight  $\lambda(x)$, which yields the orthogonal expansion
\begin{equation}\label{eq:inf_proj}
\sqrt{p(x)} = \sum_{j = 0}^{\infty} a_j \phi_j(x), \quad a_j = \int_{-\infty}^\infty \sqrt{p(x)} \phi_j(x) \lambda(x) \dd x.
\end{equation}
This way, the approximation is determined by the coefficients $(a_j)_{j=0}^{\infty}$.

Now we consider the pullback approach ({\romannumeral 2}). To apply the square-root approximation procedure, we express the pullback density $f_Z(z) = f_X(x(z))x'(z)$ as
\begin{equation}\label{eq:bouned_ratio}
    f_Z(z) = q(z) \lambda^{\rm ref}(z), \quad q(z) = \frac{f_Z(z)}{\lambda^{\rm ref}(z)}.
\end{equation}
Then, using the basis functions $(\psi_j)_{j=0}^\infty$ and the weight function $\lambda^{\rm ref}(z)$ on $(-1,1)$, the function $\sqrt{q(z)}$ has an orthogonal expansion 
\begin{equation}\label{eq:bounded_proj}
    \sqrt{q(z)} = \sum_{j = 0}^\infty b_j \psi_j(z), \quad b_j =  \int_{-1}^1 \sqrt{q(z)} \psi_j(z) \lambda^{\rm ref}(z) \dd z .
\end{equation}

Note that the function $q(z)$ can also be expressed as
\[
    q(z) = \frac{f_Z(z)}{\lambda^{\rm ref}(z)} =  \frac{f_X(x(z))x'(z)}{\lambda^{\rm ref}(z)} \frac{\lambda(x(z)) }{\lambda(x(z)) } = \frac{ p(x(z))   }{ \lambda^{\rm ref}(z) } \lambda(x(z)) x'(z)
\]
where $p$ is defined in \eqref{eq:inf_ratio}. By the definition of the induced weight function $\lambda(x)$ and the invertiblity of the map $z'(x)$, we have $\lambda^{\rm ref}(z) = \lambda(x(z)) x'(z)$, and thus $q(z) = p(x(z))$. This way, the coefficients $(a_j)_{j=0}^{\infty}$ can be expressed as
\begin{align*}
a_j  =  \int_{-\infty}^\infty \sqrt{p(x)} \psi_j(z(x)) \lambda^{\rm ref}(z(x)) z'(x) \dd x = \int_{-1}^{1} \sqrt{q(z)} \psi_j(z) \lambda^{\rm ref}(z) \dd z = b_j.
\end{align*}
Thus, the coefficient $(a_j)_{j=0}^\infty$ in \eqref{eq:inf_proj} and $(b_j)_{j=0}^\infty$ in \eqref{eq:bounded_proj} are identical.

\bibliographystyle{siamplain}
\bibliography{references}

\begin{thebibliography}{10}

\bibitem{akhiezer1965classical}
{\sc N.~I. Akhiezer and N.~Kemmer}, {\em The classical moment problem: and some
  related questions in analysis}, vol.~5, Oliver \& Boyd Edinburgh, 1965.

\bibitem{baptista2020adaptive}
{\sc R.~Baptista, Y.~Marzouk, and O.~Zahm}, {\em On the representation and
  learning of monotone triangular transport maps}, arXiv preprint
  arXiv:2009.10303,  (2020).

\bibitem{beskos2016convergence}
{\sc A.~Beskos, A.~Jasra, N.~Kantas, A.~Thiery, et~al.}, {\em On the
  convergence of adaptive sequential monte carlo methods}, Annals of Applied
  Probability, 26 (2016), pp.~1111--1146.

\bibitem{boyd2001chebyshev}
{\sc J.~P. Boyd}, {\em Chebyshev and Fourier spectral methods}, Courier
  Corporation, 2001.

\bibitem{bigoni2019greedy}
{\sc M.~Brennan, D.~Bigoni, O.~Zahm, A.~Spantini, and Y.~Marzouk}, {\em Greedy
  inference with structure-exploiting lazy maps}, NeurIPS, 33 (2020),
  pp.~8330--8342.

\bibitem{buhmann2003radial}
{\sc M.~D. Buhmann}, {\em Radial basis functions: theory and implementations},
  vol.~12, Cambridge university press, 2003.

\bibitem{caterini2021variational}
{\sc A.~Caterini, R.~Cornish, D.~Sejdinovic, and A.~Doucet}, {\em Variational
  inference with continuously-indexed normalizing flows}, in UAI, 2021,
  pp.~44--53.

\bibitem{chen2017hessian}
{\sc P.~Chen, U.~Villa, and O.~Ghattas}, {\em Hessian-based adaptive sparse
  quadrature for infinite-dimensional bayesian inverse problems}, Comput.
  Methods in Appl. Mech. Eng., 327 (2017), pp.~147--172.

\bibitem{chen2019residualflows}
{\sc R.~T.~Q. Chen, J.~Behrmann, D.~K. Duvenaud, and J.-H. Jacobsen}, {\em
  Residual flows for invertible generative modeling}, in NeurIPS, vol.~32,
  2019.

\bibitem{chkifa2015breaking}
{\sc A.~Chkifa, A.~Cohen, and C.~Schwab}, {\em Breaking the curse of
  dimensionality in sparse polynomial approximation of parametric {PDE}s},
  Journal de Math{\'e}matiques Pures et Appliqu{\'e}es, 103 (2015),
  pp.~400--428.

\bibitem{cohen2017optimal}
{\sc A.~Cohen and G.~Migliorati}, {\em Optimal weighted least-squares methods},
  The SMAI journal of computational mathematics, 3 (2017), pp.~181--203.

\bibitem{cohen2018multivariate}
{\sc A.~Cohen and G.~Migliorati}, {\em Multivariate approximation in downward
  closed polynomial spaces}, in Contemporary Computational Mathematics-A
  celebration of the 80th birthday of Ian Sloan, Springer, 2018, pp.~233--282.

\bibitem{conrad2016accelerating}
{\sc P.~R. Conrad, Y.~M. Marzouk, N.~S. Pillai, and A.~Smith}, {\em
  Accelerating asymptotically exact mcmc for computationally intensive models
  via local approximations}, J. Am. Stat. Assoc., 111 (2016), pp.~1591--1607.

\bibitem{cui2021deep}
{\sc T.~Cui and S.~Dolgov}, {\em Deep composition of tensor-trains using
  squared inverse rosenblatt transports}, Found. Comput. Math., 22 (2022),
  pp.~1863--1922.

\bibitem{cui2021conditional}
{\sc T.~Cui, S.~Dolgov, and O.~Zahm}, {\em Scalable conditional deep inverse
  rosenblatt transports using tensor-trains and gradient-based dimension
  reduction}, arXiv:2106.04170,  (2023).

\bibitem{cui2021data}
{\sc T.~Cui and O.~Zahm}, {\em Data-free likelihood-informed dimension
  reduction of bayesian inverse problems}, Inverse Probl., 37 (2021),
  p.~045009.

\bibitem{daubechies1992ten}
{\sc I.~Daubechies}, {\em Ten lectures on wavelets}, SIAM, 1992.

\bibitem{de1993construction}
{\sc C.~de~Boor, R.~A. DeVore, and A.~Ron}, {\em On the construction of
  multivariate (pre) wavelets}, Constructive approximation, 9 (1993),
  pp.~123--166.

\bibitem{devore1988interpolation}
{\sc R.~A. DeVore and V.~A. Popov}, {\em Interpolation of besov spaces},
  Transactions of the American Mathematical Society, 305 (1988), pp.~397--414.

\bibitem{dohler2010nonequispaced}
{\sc M.~D{\"o}hler, S.~Kunis, and D.~Potts}, {\em Nonequispaced hyperbolic
  cross fast fourier transform}, SIAM journal on numerical analysis, 47 (2010),
  pp.~4415--4428.

\bibitem{dafs-tt-bayes-2019}
{\sc S.~Dolgov, K.~Anaya-Izquierdo, C.~Fox, and R.~Scheichl}, {\em
  Approximation and sampling of multivariate probability distributions in the
  tensor train decomposition}, Stat. Comput., 30 (2020), pp.~603--625.

\bibitem{dung2018hyperbolic}
{\sc D.~D{\~u}ng, V.~Temlyakov, and T.~Ullrich}, {\em Hyperbolic cross
  approximation}, Springer, 2018.

\bibitem{DGKP-SEIR-2021}
{\sc R.~Dutta, S.~N. Gomes, D.~Kalise, and L.~Pacchiardi}, {\em Using mobility
  data in the design of optimal lockdown strategies for the {COVID-19}
  pandemic}, {PLoS} Comput. Biol., 17 (2021), pp.~1--25.

\bibitem{gelman1998simulating}
{\sc A.~Gelman and X.-L. Meng}, {\em Simulating normalizing constants: From
  importance sampling to bridge sampling to path sampling}, Stat. Sci.,
  (1998), pp.~163--185.

\bibitem{gradinaru2007fourier}
{\sc V.~Gradinaru}, {\em Fourier transform on sparse grids: Code design and the
  time dependent schr{\"o}dinger equation}, Computing, 80 (2007), pp.~1--22.

\bibitem{hampton2015coherence}
{\sc J.~Hampton and A.~Doostan}, {\em Coherence motivated sampling and
  convergence analysis of least squares polynomial chaos regression}, Comput.
  Methods in Appl. Mech. Eng., 290 (2015), pp.~73--97.

\bibitem{huan2015numerical}
{\sc X.~Huan}, {\em Numerical approaches for sequential Bayesian optimal
  experimental design}, PhD thesis, Massachusetts Institute of Technology,
  2015.

\bibitem{kantas2014sequential}
{\sc N.~Kantas, A.~Beskos, and A.~Jasra}, {\em Sequential monte carlo methods
  for high-dimensional inverse problems: A case study for the navier--stokes
  equations}, SIAM/ASA J. Uncertain. Quantif., 2 (2014), pp.~464--489.

\bibitem{knothe1957contributions}
{\sc H.~Knothe}, {\em Contributions to the theory of convex bodies.}, Michigan
  Math. J., 4 (1957), pp.~39--52.

\bibitem{kruse2019hint}
{\sc J.~Kruse, G.~Detommaso, U.~K{\"o}the, and R.~Scheichl}, {\em {HINT}:
  Hierarchical invertible neural transport for density estimation and bayesian
  inference}, 35 (2021), pp.~8191--8199.

\bibitem{lam2020multifidelity}
{\sc R.~Lam, O.~Zahm, Y.~Marzouk, and K.~Willcox}, {\em Multifidelity dimension
  reduction via active subspaces}, SIAM J. Sci. Comput., 42 (2020),
  pp.~A929--A956.

\bibitem{mallat1989multiresolution}
{\sc S.~G. Mallat}, {\em Multiresolution approximations and wavelet orthonormal
  bases of {$L^2(\mathbb{R})$}}, Trans. Am. Math. Soc., 315 (1989), pp.~69--87.

\bibitem{metropolis1987beginning}
{\sc N.~Metropolis}, {\em The beginning of the monte carlo method}, Los Alamos
  Science, 15 (1987), pp.~125--130.

\bibitem{migliorati2015adaptive}
{\sc G.~Migliorati}, {\em Adaptive polynomial approximation by means of random
  discrete least squares}, in Numerical Mathematics and Advanced Applications,
  Springer, 2015, pp.~547--554.

\bibitem{migliorati2019adaptive}
{\sc G.~Migliorati}, {\em Adaptive approximation by optimal weighted
  least-squares methods}, SIAM J. Numer. Anal., 57 (2019), pp.~2217--2245.

\bibitem{el2012Bayesian}
{\sc T.~Moselhy and Y.~Marzouk}, {\em Bayesian inference with optimal maps}, J.
  Comput. Phys., 231 (2012), pp.~7815--7850.

\bibitem{myers2021sequential}
{\sc A.~Myers, A.~H. Thi{\'e}ry, K.~Wang, and T.~Bui-Thanh}, {\em Sequential
  ensemble transform for bayesian inverse problems}, Journal of Computational
  Physics, 427 (2021), p.~110055.

\bibitem{narayan2017christoffel}
{\sc A.~Narayan, J.~Jakeman, and T.~Zhou}, {\em A christoffel function weighted
  least squares algorithm for collocation approximations}, Math. Comput., 86
  (2017), pp.~1913--1947.

\bibitem{Owen_2013}
{\sc A.~B. Owen}, {\em Monte Carlo theory, methods and examples}, 2013.

\bibitem{papamakarios2021normalizing}
{\sc G.~Papamakarios, E.~Nalisnick, D.~J. Rezende, S.~Mohamed, and
  B.~Lakshminarayanan}, {\em Normalizing flows for probabilistic modeling and
  inference}, J. Mach. Learn Res., 22 (2021), pp.~1--64.

\bibitem{parno2018transport}
{\sc M.~D. Parno and Y.~M. Marzouk}, {\em Transport map accelerated markov
  chain monte carlo}, SIAM/ASA J. Uncertain. Quantif., 6 (2018), pp.~645--682.

\bibitem{rosenblatt1952remarks}
{\sc M.~Rosenblatt}, {\em Remarks on a multivariate transformation}, Ann. Math.
  Stat., 23 (1952), pp.~470--472.

\bibitem{schillings2016scaling}
{\sc C.~Schillings and C.~Schwab}, {\em Scaling limits in computational
  bayesian inversion}, ESAIM Math. Model. Numer. Anal., 50 (2016),
  pp.~1825--1856.

\bibitem{IP:SchStu_2012}
{\sc C.~Schwab and A.~M. Stuart}, {\em Sparse deterministic approximation of
  {B}ayesian inverse problems}, Inverse Probl., 28 (2012), p.~045003.

\bibitem{shen2011spectral}
{\sc J.~Shen, T.~Tang, and L.-L. Wang}, {\em Spectral methods: algorithms,
  analysis and applications}, vol.~41, Springer Science \& Business Media,
  2011.

\bibitem{shen2010sparse}
{\sc J.~Shen and L.-L. Wang}, {\em Sparse spectral approximations of
  high-dimensional problems based on hyperbolic cross}, SIAM J. Numer. Anal.,
  48 (2010), pp.~1087--1109.

\bibitem{shen2014approximations}
{\sc J.~Shen, L.-L. Wang, and H.~Yu}, {\em Approximations by orthonormal mapped
  chebyshev functions for higher-dimensional problems in unbounded domains}, J.
  Comput. Appl. Math., 265 (2014), pp.~264--275.

\bibitem{smetana2020randomized}
{\sc K.~Smetana and O.~Zahm}, {\em Randomized residual-based error estimators
  for the proper generalized decomposition approximation of parametrized
  problems}, Int. J. Numer. Methods Eng., 121 (2020), pp.~5153--5177.

\bibitem{spantini2018inference}
{\sc A.~Spantini, D.~Bigoni, and Y.~Marzouk}, {\em Inference via
  low-dimensional couplings}, J. Mach. Learn Res., 19 (2018), pp.~2639--2709.

\bibitem{villani2008optimal}
{\sc C.~Villani}, {\em Optimal transport: old and new}, vol.~338, Springer
  Science \& Business Media, 2008.

\bibitem{wang2022minimax}
{\sc S.~Wang and Y.~Marzouk}, {\em On minimax density estimation via measure
  transport}, arXiv preprint arXiv:2207.10231,  (2022).

\bibitem{wendland2004scattered}
{\sc H.~Wendland}, {\em Scattered data approximation}, vol.~17, Cambridge
  university press, 2004.

\bibitem{williams2006gaussian}
{\sc C.~K. Williams and C.~E. Rasmussen}, {\em Gaussian processes for machine
  learning}, vol.~2, MIT press Cambridge, MA, 2006.

\bibitem{xiu2015stochastic}
{\sc D.~Xiu}, {\em Stochastic collocation methods: a survey}, Handbook of
  uncertainty quantification,  (2015), pp.~1--18.

\bibitem{xiu2002wiener}
{\sc D.~Xiu and G.~E. Karniadakis}, {\em The wiener--askey polynomial chaos for
  stochastic differential equations}, SIAM J. Sci. Comput., 24 (2002),
  pp.~619--644.

\bibitem{yan2019adaptive}
{\sc L.~Yan and T.~Zhou}, {\em Adaptive multi-fidelity polynomial chaos
  approach to bayesian inference in inverse problems}, Journal of Computational
  Physics, 381 (2019), pp.~110--128.

\bibitem{zech2022sparsea}
{\sc J.~Zech and Y.~Marzouk}, {\em Sparse approximation of triangular
  transports, part i: The finite-dimensional case}, Constr. Approx.,  (2022),
  pp.~1--68.

\end{thebibliography}

\end{document}